\documentclass{amsart}

\usepackage{graphicx,epsfig}

\usepackage{amsmath,amssymb,latexsym, amsfonts, amscd, amsthm, xy}

\usepackage{url}

\usepackage{bbm}

\usepackage{hyperref}

\usepackage{mathrsfs}

\usepackage[margin=1in]{geometry}

\input{xy}
\xyoption{all}

\usepackage[usenames]{color}

\hypersetup{
    colorlinks   = true,
      citecolor    = red
}

\hypersetup{linkcolor=red}

\hypersetup{linkcolor=blue}

\theoremstyle{plain}
\newtheorem{theorem}{Theorem}[section]
\newtheorem{conjecture}[theorem]{Conjecture}
\newtheorem{proposition}[theorem]{Proposition}
\newtheorem{corollary}[theorem]{Corollary}

\newtheorem{lemma}[theorem]{Lemma}

\theoremstyle{definition}
\newtheorem{definition}[theorem]{Definition}
\newtheorem{remark}[theorem]{Remark}
\newtheorem{example}[theorem]{Example}

\def\bC{\mathbb{C}}

\def\bH{\mathbb{H}}

\def\bZ{\mathbb{Z}}
\def\bR{\mathbb{R}}
\def\bS{\mathbb{S}}

\def\cD{\mathcal{D}}

\def\cG{\mathcal{G}}

\def\cH{\mathcal{H}}

\def\cH{\mathcal{H}}

\def\cI{\mathcal{I}}

\def\cJ{\mathcal{J}}

\def\cL{\mathcal{L}}
\def\cM{\mathcal{M}}

\def\cO{\mathcal{O}}

\def\cP{\mathcal{P}}

\def\cQ{\mathcal{Q}}

\def\cR{\mathcal{R}}

\def\cS{\mathcal{S}}

\def\cX{\mathcal{X}}

\def\cZ{\mathcal{Z}}

\def\sord{\textbf{sord}}

\def\ord{\textbf{ord}}

\def\fd{\mathfrak{d}}

\def\Div{\textbf{Div}}

\def\divi{\textbf{div}}

\def\ord{\textbf{ord}}

\def\b1{\mathbbm{1}}

\begin{document}

\title{Quaternionic analogues of Bers's theorem and Iss'sa's theorem}

\author{Dong Quan Ngoc Nguyen}

\date{July 25, 2021}

\address{Department of Applied and Computational Mathematics and Statistics \\
         University of Notre Dame \\
         Notre Dame, Indiana 46556, USA }

\email{\href{mailto:dongquan.ngoc.nguyen@nd.edu}{\tt dongquan.ngoc.nguyen@nd.edu}}

\urladdr{http://nd.edu/~dnguye15}



\maketitle

\tableofcontents

\begin{abstract}

In their recent work, Gentili and Struppa proposed a different quaternionic analogue of the notion of holomorphic functions in the complex plane, called \textit{slice regular functions}, which has led to several analogues of classical theorems in complex function theory in the quaternionic setting. The quaternionic analogue of meromorphic functions is called \textit{semiregular functions}. Such a function is said to be \textit{slice preserving} if it maps each domain in a complex line to a domain in the same complex line.  In this work, we deal with the relation between analytic and algebraic properties of certain fields of semiregular functions in symmetric slice domains of the real quaternions. More precisely, we prove quaternionic analogues of Bers's theorem and Iss'sa's theorem that deal with the algebraic structures of the ring of slice preserving regular functions and of certain fields of slice preserving semiregular functions on symmetric slice domains in the real quaternions.

\end{abstract}

\section{Introduction}

Complex function theory (see \cite{rem1, rem2}) is a highly developed area of mathematics, which has applications in a diverse variety of other areas in mathematics. Since the discovery of quaternions by William Rowan Hamilton in the 19th century, several mathematicians have been searching for a quaternionic analogue of complex function theory. The first attempt, due to Fueter \cite{fueter, fueter1, sudbery}, proposed a quaternionic analogue, called the \textbf{Cauchy-Fueter operator}, of the Cauchy-Riemann operator which shares several analogies with the theory of holomorphic functions such as a quaternionic analogue of Cauchy integral theory, and the Taylor and Laurent series for functions satisfying the Cauchy-Fueter operator. Despite its richness, and important applications to physics, the set consisting of all functions satisfying the Cauchy-Fueter operator, called \textbf{Fueter-regular functions}, contains several strange properties that are not in analogy with classical holomorphic functions in the complex plane $\bC$. For example, the identity function is not regular in the sense of Fueter, which implies all quaternionic polynomials are not Fueter-regular. Such exclusion of an important class of functions, say quaternionic polynomials, limits applications of the theory proposed by Fueter, and motivates the search for possibly different quaternionic analogues of complex function theory.

Most notably, in their recent work, Gentili and Struppa \cite{genstru, genstru1} proposed a different quaternionic analogue of the notion of holomorphic functions, called \textbf{slice regular functions}, which has led to several analogues of classical theorems in complex function theory in the quaternionic setting. One of the important features of Gentili and Struppa's theory is that the set of all slice regular functions is equipped with a (noncommutative) \textbf{ring structure} (see Lang \cite{lang} for such notion) with respect to the usual addition ``$+$" and the $\star$-product proposed by Colombo, Gentili, Sabadini, and Struppa \cite{cgss} that will be reviewed in Section \ref{sec-notions}. Such ring structure shows that slice regular functions operates under the usual addition and the $\star$-product in a similar way as holomorphic functions do under the usual addition and multiplication, which allows to search for possible relations between analysis, algebra, and geometry of slice regular functions. In addtion, the notion of slice regular functions proposed by Gentili and Struppa leads to many fundamental analogues of the classical theorems in complex function theory in the quaternionic setting such as the quaternionic analogues of Cauchy and Pompeiu representation formulae, Cauchy inqualities, the maximum modulus principle, the open mapping theorem, Laurent series, and the Weierstrass factorization theorem, to cite just a few (see, for example, \cite{cgss, cgss1, cgs, genstop, gv, stoppato}). Especially, the reader is referred to a beautiful exposition of theory of regular functions of a quaternionic variable by Gentili, Stoppato, and Struppa \cite{GSS}. 

Although theory of slice regular functions is highly developed, there seems a lack of results that study regular functions from an \textit{algebraic} point of view. The aim of this paper is to explore the algebraic structure of semiregular functions (which are analogous to mermorphic functions in the complex case), and study the relation between such algebraic structure and the analytic information of semiregular functions such as their poles and zeros. More precisely, in this paper, we prove a quaternionic analogue of Bers's and Iss'sa's theorems that describe \textit{valuations} on the ring of slice preserving regular functions and its quotient field. Valuations play a role of ordering elements in an algebraic structure such as groups, rings, or fields (see Lang \cite{lang}), and it is one of the main notions in algebra that has important applications in algebra, algebraic geometry, and number theory.

The structure of our paper is as follows. In Section \ref{sec-notions}, we recall some basic notions and notation in theory of regular functions of a quaternionic variable that will be used in the subsequent sections. In Section \ref{sec-divisors}, we introduce a notion of spherical divisors in a symmetric slice domain in the quaternions $\bH$, in analogy with the notion of divisors in $\bC$. We introduce a generalization of the notion of slice preserving regular functions that we call \textit{slice preserving semiregular functions}. In the same section, we also introduce a notion of spherical orders of slice preserving semiregular functions that plays a key role in our main theorems. At the end of Section \ref{sec-divisors}, we prove a quaternipnic analogue of the Holomorphy Criterion for slice preserving semiregular functions. In Section \ref{sec-weierstrass}, we recall a quaternionic analogue of the Weierstrass factorization theorem, due to Gentili and Vignozzi \cite{gv}. Using this theorem, we prove that the field of slice preserving semiregular functions in $\bH$  is the quotient field of slice preserving entire functions in $\bH$. The set of all slice preserving regular functions is an integral domain, and can be equipped with a commutative algebra structure over $\bR$. In Section \ref{sec-bers}, we prove a quaternionic analogue of Bers's theorem that studies valuations on the algebra of slice preserving regular functions. In Section \ref{sec-isssa}, we prove our main theorem in this paper which is a quaternionic analogue of Iss'sa's theorem, regarding valuations on the quotient field of the algebra of slice preserving regular functions. One of the main results used in the proof of Iss'sa's theorem in the complex case (see Iss'sa \cite{isssa}, \cite{kelleher}, or Remmert \cite{rem2}) is the Root Criterion for the existence of a holomorphic $n$th root of a holomorphic function for every $n \ge 1$ (see \cite{rem2}). There is, however, no known analogue of the Root Criterion for regular functions in the quaternionic setting. Thus in the proof of our quaternionic analogue of Iss'sa's theorem, we need to get round to this issue by explicitly proving that certain functions used in the proof of the quaternionic analogue of Iss'sa's theorem have $n$th roots for arbitrary positive integers $n$.

\section{Some basic notions and notation}
\label{sec-notions}

Let $\bH$ denote the division algebra of quaternions that is generated as a 4-dimensional vector space over $\bR$ by the elements $1, i, j, k$ such that $i^2 = j^2 = k^2 = -1$, $ij = -ji$, $jk = -kj$, and $ki = -ik$. Each element $q$ in $\bH$ can be written in the form $q = a_0 + a_1 i + a_2 j + a_3 k$ for some real numbers $a_0, a_1, a_2, a_3$. The \textit{real part of $q$ }is $a_0$, and we denote it by $\Re(q)$. The \textit{imaginary 
part} is $a_1i + a_2 j + a_3k$. We denote the imaginary part of $q$ by $\Im(q)$. The \textit{modulus of $q$}, denoted by $|q|$, is defined as the Euclidean norm of $q$ in $\bR^4$, i.e., $|q| = \sqrt{a_0^2 + a_1^2 + a_2^2 + a_3^2}$. So each quaternion $q$ can be written in the form $q = \Re(q) + \Im(q)$. The \textit{conjugate of $q$} is defined as $\bar{q} = \Re(q) - \Im(q)$. 

Let $\bS$ denote the unit sphere of purely imaginary quaternions, consisting of all elements $q \in \bH$ such that $q^2 = -1$. Equivalently, it is easy to verify that
\begin{align*}
\bS = \{q \in \bH \; | \; \text{$q = a_1 i + a_2 j + a_3 k$ and $a_1^2 + a_2^2 + a_3^2 = 1$}\}.
\end{align*}

Let $q = a_0 + a_1i + a_2j + a_3k \in \bH$ for some real numbers $a_0, a_1, a_2, a_3$ such that $\Im(q) = a_1 i + a_2 j + a_3k \ne 0$. Set 
\begin{align*}
I = \dfrac{\Im(q)}{|\Im(q)|} = \dfrac{a_1}{(a_1^2 + a_2^2 + a_3^2)^{1/2}}i +  \dfrac{a_2}{(a_1^2 + a_2^2 + a_3^2)^{1/2}}j +  \dfrac{a_3}{(a_1^2 + a_2^2 + a_3^2)^{1/2}}k.
\end{align*}
It is immediate to see that $I \in \bS$, and $q$ can be represented in the form
\begin{align}
\label{e-rep-of-q-in-S}
q = x + Iy,
\end{align}
where $x = \Re(q)$ and $y = |\Im(q)|$. In addition, $\bar{q} = x - Iy$, which is a quaternionic analogue of the complex conjugate in the complex case. From the representation (\ref{e-rep-of-q-in-S}) of $q$, we deduce that
\begin{align*}
|q| = \sqrt{x^2 + y^2}.
\end{align*}

For each $I \in \bS$, the complex line $\bC_I$ is defined as $\bC_I = \bR + I\bR$ which is isomorphic to the complex plane $\bC$. One can write
\begin{align*}
\bH = \bigcup_{I \in \bS} \bC_I.
\end{align*}

Let $f : D \to \bH$ be a function on a domain $D$ in $\bH$. For each $I \in \bH$, we denote by $f_I$ the restriction of $f$ to $D \cap \bC_I$, i.e., for all $q \in D \cap \bC_I$, $f_I(q) = f(q)$. One can represent each element $q$ in $D \cap \bC_I$ as $x + Iy$ for $x, y \in \bR$, and thus the restriction $f_I$ can be viewed as a function in two real variables $x, y$.

We recall from \cite{cgss} the notion of regular functions.
\begin{definition}
(regular functions)

Let $D$ be a domain in $\bH$, and let $f : D \to \bH$ be a function on $D$ that takes values in $\bH$. The function $f$ is said to be \textbf{regular on $D$} if for every element $I \in \bS$, the restriction $f_I$ of $f$ to $D \cap \bC_I$ has continuous partial derivatives $\dfrac{\partial f}{\partial x}, \dfrac{\partial f}{\partial y}$ such that
\begin{align*}
\bar{\partial}_I f(x + Iy) := \dfrac{1}{2}\left(\dfrac{\partial f}{\partial x} + I \dfrac{\partial f}{\partial y} \right)f_I(x + Iy) = 0
\end{align*}
for all $x + Iy \in D \cap \bC_I$.

\end{definition}

It is known (see \cite{cgss}) that there exist regular functions on certain domains in $\bH$ such that these functions are not continuous. So a notion of domains in $\bH$, using topology induced from $\bR^4$ is not natural to study the class of regular functions. In \cite{cgss}, a special class of domains in $\bH$ was introduced to get round to this issue as follows.

\begin{definition}
(slice domains)

Let $D$ be a domain in $\bH$ such that $D \cap \bR$ is nonempty. The domain $D$ is said to be a \textbf{slice domain} if for every $I \in \bS$, $D \cap \bC_I$ is a domain in the complex line $\bC_I$.

\end{definition}

For each $x + yI$ in $\bH$ for some $I \in \bS$, we denote by $[x + y\bS]$ the sphere $\{x + yJ \; | \; J \in \bS\}$ in $\bH$. We also allow $y = 0$, and thus in this case, $[x + \bS y]$ is a singleton set $\{x\}$. Since $\bH$ is a union of complex lines $\bC_I$, the following notion is natural for subsets of $\bH$.
\begin{definition}
(symmetric domains)

Let $X$ be a subset of $\bH$. The set $X$ is said to be \textbf{symmetric} if the following is true:
\begin{itemize}

\item [($\star$)] for every element $x + Iy \in X$ for some real numbers $x, y$ and $I \in \bS$, the whole sphere $[x + y\bS]$ is contained in $X$.

\end{itemize}

\end{definition}

In \cite{cgss}, Colombo, Gentili Sabadini, Struppa proved that a regular function on a slice domain $D$ can be uniquely extended to the \textit{smallest symmetric slice domain} that contains $D$. This result shows that symmetric slice domains are quaternionic analogues of domains of holomorphy in complex function theory.

Regular functions possess a \textit{splitting property} that can be used to study their restrictions to the complex lines $\bC_I$. Before stating this property, we fix some notation that will be used throughout the paper. Let $I, J$ be elements in $\bS$ of the form $I = a_1 i + a_2j + a_3k$ and $J = b_1i + b_2j + b_3k$ for some real numbers $a_1, a_2, a_3, b_1, b_2, b_3$. We write $\langle I, J \rangle$ for the inner product of $I, J$ as vectors in $\bR^4$, that is, 
\begin{align*}
\langle I, J \rangle = a_1b_1 + a_2b_2 + a_3b_3.
\end{align*}

We write $I \perp J$ whenever $I, J$ are orthogonal as vectors in $\bR^4$, that is, whenever $\langle I, J\rangle = 0$. One can view $I, J$ as vectors $(a_1, a_2, a_3)$ and $(b_1, b_2, b_3)$ in $\bR^3$, respectively with respect to the basis $(i, j, k)$. Based on this view, we can define a cross product $I \times J \in \bR^3$ between $I$ and $J$ of the form
\begin{align*}
I \times J = 
\begin{vmatrix}
i & j & k \\
a_1 & a_2 & a_3 \\
b_1 & b_2 & b_3
\end{vmatrix}
.
\end{align*}

The following result is immediate.
\begin{proposition}
\label{prop-inner-cross-prod-relations}
Let $I, J$ be elements in $\bS$. Then

\begin{itemize}

\item []

\item [(i)] $IJ = -\langle I, J \rangle + I \times J$.

\item [(ii)] If $I \perp J$, then $IJ = I \times J$ belongs in $\bS$, and $1, I, J, K := IJ$ forms a basis of $\bH$ that satisfies the same algebraic properties as the standard basis $1, i, j, k$.

\end{itemize}

\end{proposition}

The following is a splitting property for regular functions.

\begin{lemma}
\label{lem1.3-gss}
(See \cite{cgss} or \cite[Lemma 1.3, p.2]{GSS})

Let $f$ be a regular function on a slice domain $D$. Let $I, J \in \bS$ such that $I \perp J$. Then there exist holomorphic functions $F, G : D \cap \bC_I \to \bC_I$ such that
\begin{align*}
f_I(z) = F(z) + J G(z)
\end{align*}
for all $z \in D \cap \bC_I$.

\end{lemma}

Using the Splitting Lemma, a quaternionic analogue of the Identity Principle for regular functions was proved in \cite{cgss}.

\begin{theorem}
\label{thm1.12-gss}
(see \cite{cgss} or \cite[Theorem 1.12]{GSS})
(Identity Principle)
Let $f, g$ be regular functions on a slice domain $D$. Assume that there exists an element $I \in \bS$ such that the set $\{z \in D \cap \bC_I \; | \; f(z) = g(z)\}$ contains an accumulation point in $D \cap \bC_I$. Then $f = g$ on $D$.

\end{theorem}

The values of regular functions in distinct complex lines are related in the following.
\begin{theorem}
\label{cor-1.16-gss}
(See \cite{cgss} or \cite[ Corollary 1.16]{GSS})

Let $f$ be a regular function on a symmetric slice domain $D$ and let $x + y\bS \subset D$. For all $I, J \in \bS$,
\begin{align*}
f(x + yJ) = \dfrac{1}{2}(f(x + yI) + f(x - yI)) + \dfrac{JI}{2}(f(x - yI) - f(x + yI)).
\end{align*}

\end{theorem}

It is known that pointwise product of regular functions is in general not regular. Motivated by the multiplication between polynomials with coefficients in a noncommutative algebra (see \cite{lam}), Gentili and Stoppato \cite{genstop1} extended this notion of multiplication for regular functions defined on the open unit ball of $\bH$ in terms of their power series expansion, which they call the \textit{$\star$-product}. The notion of $\star$-product was generalized to regular functions on an arbitrary symmetric slice domain in \cite{cgss}, based on the extension lemma below that expresses the relation between values of regular function in different complex lines.

\begin{lemma}
\label{lem-extension}
(see \cite{cgss})

Let $D$ be a symmetric slice domain, and let $I$ be an element in $\bS$. Let $f_I : D \cap \bC_I \to \bH$ be a homolomorphic function on $D \cap \bC_I$. For each $J \in \bS$, define
\begin{align*}
f(x + Jy)  := \dfrac{1}{2}(f_I(x + yI) + f_I(x - yI)) + \dfrac{JI}{2}(f_I(x - yI) - f(x + yI))
\end{align*}
for all $x + Jy \in D$. Then $f : D \to \bH$ is a regular function that coincides with $f_I$ on $D \cap \bC_I$, i.e., $f$ extends $f_I$ to the whole domain $D$. Furthermore, $f$ is the unique extension of $f_I$ to $D$, and we denote it by $\text{ext}(f_I)$.

\end{lemma}

We now recall the notion of the $\star$-product of regular functions. Let $f, g$ be regular functions on a symmetric slice domain $D$. Let $I, J \in \bS$ such that $I, J$ are orthogonal as vectors in $\bR^4$. Using Lemma \ref{lem1.3-gss}, there exist holomorphic functions $F, G, H, K : D \cap \bC_I \to \bC_I$ such that 
\begin{align*}
f_I(z) &= F(z) + G(z)J, \\
 g_I(z) &= H(z) + K(z) J
 \end{align*}
 for all $z \in D \cap \bC_I$.
 
 For every $z \in D \cap \bC_I$, define
 \begin{align*}
 f_I \star g_I(z) := (F(z)H(z) - G(z)\overline{K(\bar{z})}) + (F(z)K(z) + G(z)\overline{H(\bar{z})})J.
 \end{align*}
It is clear that $f_I\star g_I : D \cap \bC_I \to \bH$ is a holomorphic function, and it thus follows from Lemma \ref{lem-extension} that there exists a unique extension of $f_I \star g_I$, denoted by $\text{ext}(f_I \star g_I)$, to the whole domain $D$. 
\begin{definition}
\label{def-regular-product}

The function $f \star g : D \to \bH$ defined by 
\begin{align*}
f \star g(q) := \text{ext}(f_I \star g_I)(q)
\end{align*}
for every $q \in D$, is called the \textbf{$\star$-product of $f$ and $g$}.

\end{definition}

An important class of regular functions considered in our paper is slice preserving regular functions.
\begin{definition}
\label{def-slice-preserving-reg-func}

A regular function $f$ on a symmetric slice domain $D$ is said to be \textbf{slice preserving} if for every $I \in \bS$, $f(D \cap \bC_I) \subset \bC_I$, that is, the restriction of $f$ to each complex line $\bC_I$ is a holomorphic function taking values in the same complex line.

Throughout this paper, we denote by $\cO_s(D)$ the set of all slice preserving regular functions on a symmetric slice domain $D$.

\end{definition}

The following result is well-known (see \cite[Lemma 1.30]{GSS}).

\begin{lemma}
\label{lem-commutative-O_s}

Let $f, g$ be regular functions on a symmetric slice domain $D$ such that $f$ is slice preserving. Then $fg$ is regular on $D$ and
\begin{align*}
f \star g = g \star f = fg.
\end{align*}

\end{lemma}

\begin{proof}

The fact that $fg$ is regualr and that $f\star g = fg$ follows from the proof of  \cite[Lemma 1.30]{GSS}. Since the restriction $f_I$ of $f$ is a holomorphic function from $D \cap \bC_I$ into $\bC_I$, $f_I$ commutes with the restriction $g_I$ of $g$, where $I$ is arbitrary in $\bS$. By the Identity Principle (see Theorem \ref{thm1.12-gss}), $f \star g = g \star f$.

\end{proof}

The composition of regular functions is not in general regular. For slice preserving regular functions, we recall the following result that will be useful in the last section.
\begin{lemma}
\label{lem-comp-of-slice-reg-func}
(see \cite[Lemma 1.32]{GSS})

Let $D, D'$ be symmetric slice domains in $\bH$. Let $f : D \to D'$ and $g : D' \to \bH$ be regular functions such that $f$ is a slice preserving function. Then the composition $g \circ f$ is regular. Furthermore if $g$ is also a slice preserving function, then $g \circ f$ is a slice preserving regular function.

\end{lemma}

We will need the following results, regarding zeros of regular functions and their $\star$-products.

\begin{lemma}
\label{lem-spherical-multiplicity}
(see \cite{pereira, ps, ser-siu})

Let $f$ be a regular function on a symmetric slice domain $D$. 
\begin{itemize}

\item [(i)] A point $p$ in $D$ is a zero of $f$ if and only if there exists a regular function $g : D \to \bH$ such that 
\begin{align*}
f(q) = (q - p)\star g(q).
\end{align*}

\item [(ii)] $f$ vanishes identically on $[x + \bS y]$ for some real numbers $x, y \in \bR$ with $y \ne 0$ if and only if there exists a regular function $g : D \to \bH$ such that
\begin{align*}
f(q) = ((q - x)^2 + y^2)g(q).
\end{align*}

\end{itemize}

\end{lemma}

\begin{lemma}
\label{lem-zeros-of-product-of-func}
(see \cite{cgss, genstop1, genstop2})

Let $f, g$ be regular functions on a symmetric slice domain $D$, and let $q \in D$. Then $f \star g (q) = 0$ if and only if either $f(q) = 0$ or $f(q) \ne 0$ and $g(f(q)^{-1}q f(q)) = 0$.

\end{lemma}

For the rest of this section, we describe a quaternionic analogue of meromorphic functions in complex function theory. We begin by recalling some basic notions that lead to a notion of semiregular functions--a quaternionic analogue of meromorphic functions.

\begin{definition}

\begin{itemize}

\item []

\item [(i)] Let $\sigma : \bH \times \bH \to \bR_{\ge 0}$ be the map defined by
\begin{align*}
\sigma(p, q) = 
\begin{cases}
|p - q| \; \; \text{if $p, q$ belong in the same complex line $\bC_I$ for some $I \in \bS$,} \\
\sqrt{(\Re(p) - \Re(q))^2 + (|\Im(p)| + |\Im(q)|)^2}.
\end{cases}
\end{align*}

\item [(ii)] Let $\tau : \bH \times \bH \to \bR_{\ge 0}$ be the map defined by
\begin{align*}
\tau(p, q) = 
\begin{cases}
|p - q| \; \; \text{if $p, q$ belong in the same complex line $\bC_I$ for some $I \in \bS$,} \\
\sqrt{(\Re(p) - \Re(q))^2 + (|\Im(p)| - |\Im(q)|)^2}.
\end{cases}
\end{align*}

\end{itemize}

\end{definition}
We recall the sets of convergence of regular Laurent series in $\bH$ that were introduced in \cite{stoppato1, stoppato2}. 
\begin{definition}

Let $d_1, d_2$ be nonnegative real numbers such that $0 \le d_1 < d_2 \le +\infty$. For each $p \in \bH$, let $\Sigma(p; d_1, d_2)$ be the subset of $\bH$ defined by
\begin{align*}
\Sigma(p; d_1, d_2) = \{q \in \bH \; |\; \text{$\tau(q, p) > d_1$ and $\sigma(q, p) < d_2$}\}.
\end{align*}

\end{definition}

The following result is a quaternionic analogue of Laurent expansion for regular functions that was proved in \cite{stoppato1, stoppato2}.
\begin{theorem}
\label{thm-Laurent-expansion}
(see Stoppato \cite{stoppato1, stoppato2})

Let $f$ be a regular function on a domain $D$ of $\bH$, and let $p \in \bH$. Then there exists a sequence of quaternions $(a_n)_{n \in \bZ}$ in $\bH$ such that for all $0 \le d_1 < d_2 \le +\infty$ with $\Sigma(p; d_1, d_2) \subseteq D$, 
\begin{align*}
f(q) = \sum_{n \in \bZ}(q - p)^{\star n} a_n
\end{align*}
for all $q \in \Sigma(p; d_1, d_2)$. Here for each $n \in \bZ$,
\begin{align*}
(q - p)^{\star n} = \underbrace{(q - p)\star (q - p) \star \cdots \star (q - p)}_{\text{$n$ copies}}
\end{align*}

\end{theorem}

The series in the above theorem is called the \textit{Laurent expansion of $f$ at $p$}. The real  number $d_1$ is called the \textit{inner radius of convergence}, and the 
real number $d_2$ is called the \textit{outer radius of convergence}. 

Let $f$ be a regular function on a symmetric slice domain $D$. A quaternion $p \in \bH$ is said to be a \textit{singularity of $f$} if there exists a positive real number $d > 0$ such that $\Sigma(p; 0, d)$. In this case, $f$ admits the Laurent expansion of $f$ at $p$ of the form
\begin{align*}
f(q) = \sum_{n \in \bZ}(q - p)^{\star n} a_n
\end{align*}
for all $q \in \Sigma(p; 0, d)$, where the $a_i$ are certain quaternions. Thus $0$ is the inner radius of convergence, and $d$ is the outer radius of convergence.

In analogy with complex function theory, there are two types of singularities. A quaternion $p$ is a \textit{removable singularity} of $f$ if there exists a regular extension of $f$ to a neighborhood of $p$. Otherwise, $p$ is called a \textit{nonremovable singularity}. 

For a nonremovable singularity $p$ of $f$, Theorem \ref{thm-Laurent-expansion} implies that $f$ admits the Laurent expansion at $p$ of the form
\begin{align*}
f(q) = \sum_{n \in \bZ}(q - p)^{\star n} a_n
\end{align*}
for all $q \in \Sigma(p; 0, d)$, where the $a_n$ are quaternions in $\bH$. If there exists an integer $\ell \ge 0$ such that $a_{-n} = 0$ for all $n > \ell$, we call $p$ a \textit{pole of $f$}. By the well-ordering principle (see Lang \cite{lang}), there exists the \textit{smallest} such integer $\ell$ which we call the \textit{order of $p$}, and denote it by $\ord_f(p)$. If $p$ is not a pole of $f$, it is said to be an \textit{essential singularity of $f$}, and set $\ord_f(p) = +\infty$.

We recall an analogue of meromorphic functions in the quaternionic setting (see \cite{stoppato, stoppato1, stoppato2, GSS}).

\begin{definition}
\label{def-semi-regular}
(semiregular functions)

Let $f : D \to \bH$ be a function on a symmetric slice domain $D$. The function $f$ is said to be \textbf{semiregular on $D$} if there exists a subset $A$ of $D$ that is also a symmetric slice domain such that the following are true.
\begin{itemize}

\item [(i)] $f$ is regular on $A$; 

\item [(ii)] every point in $D \setminus A$ is either a pole or a removable singularity of $f$.

\end{itemize}

Throughout the paper, we denote by $\cM(D)$ the set of all semiregular functions on a symmetric slice domain $D$.

\end{definition}

The following result will be useful in subsequent sections.

\begin{theorem}
\label{cor5.27-gss}
(See \cite{stoppato, stoppato1, stoppato2} or \cite[Corollary 5.27]{GSS})

Let $f$ be a semiregular function on a symmetric slice domain $D$. Let $p = x + yI$ be a point in $D$ for some $x, y \in \bR$ and $I \in \bS$. Let $m = \ord_f(p)$ and $n = \ord_f(\bar{p})$. Without loss of generality, assume $m \le n$. Then there exists a unique semiregular function $g$ on $D$ such that the following are true.
\begin{itemize}

\item [(i)] $g$ has no poles in $[x + \bS y]$; and

\item [(ii)] $f(q) = ((q - x)^2 + y^2)^{-n}(q - p)^{\star (n - m)}\star g(q)$.

\end{itemize}

\end{theorem}

\section{Divisors and principal divisors}
\label{sec-divisors}

Let $D$ be a symmetric slice domain in $\bH$. One can represent $D$ in the form
\begin{align}
\label{eqn-D-rep}
D = \cup_{i \in \cI} \bS_i,
\end{align}
where for each $i \in \cI$, $\bS_i =[ x_i + y_i\bS]$ for some $x_i, y_i \in \bR$ such that $x_i + y_iI_i \in D$ for some $I_i \in \bS$, and $\bS_i \cap \bS_j = \emptyset$. We introduce a notion of a spherical divisor on a symmetric slice domain which can be viewed as a quaternionic analogue of the classicial notion of divisors on domains in the complex plane $\bC$ (see, for example, \cite[p.74]{rem2}). This notion will play a main role in the proof of a quaternionic analogue of Iss'sa's theorem that we will prove in the last section.

\begin{definition}
\label{def-spherical-divisor}

Let $D$ be a symmetric slice domain in $\bH$ that is represented in the form (\ref{eqn-D-rep}). A map $\fd : D \rightarrow \bZ$ is a spherical divisor on $D$ if the following are satisfied:
\begin{itemize}

\item[(i)] the restriction of $\fd$ to each $\bS_i$ is a constant;

\item [(ii)] for each $q \in D$ such that $\fd(q) \ne 0$, there exists a symmetric slice neighborhood of $z$, say $U_z \subset D$ such that there are only finitely many $\bS_i$ such that the restriction of $\fd$ to $U_z \cap \bS_i$ is nonzero.

\end{itemize}

\end{definition}

In view of condition (i), by abuse of notation, one can write $\fd([x_i + y_i\bS]) = \fd(x_i + y_i I)$ for any $I \in \bS$. 

Let $\Div(D)$ be the set of all spherical divisors on $D$. In view of condition (i) above, one can, without loss of generality, write a spherical divisor $\fd : D \rightarrow \bZ$ formally in the form
\begin{align*}
\fd = \sum_{i \in \cI} \fd([x_i + y_i\bS]) \cdot [x_i + y_i \bS],
\end{align*}
where 
\begin{align*}
D = \cup_{i \in \cI} [x_i + y_i\bS]
\end{align*}
and $[x_i + y_i \bS] \subset D$ for each $i \in \cI$. 

The set $\Div(D)$ becomes an abelian group with respect to natural addition between maps from $D$ to $\bZ$. The trivial map which sends every element in $D$ to $0$ is the identity element in the group $\Div(D)$.

For a symmetric slice domain $D$ in $\bH$, recall from Section \ref{sec-notions} that $\cO_s(D)$ denotes the set of all regular functions $f$ on $D$ that are also slice preserving, that is, $f$ is regular on $D$ such that $f(D \cap \bC_I) \subseteq \bC_I$ for all $I \in \bS$.

The next result is immediate from Lemma \ref{lem-commutative-O_s}.

\begin{corollary}
\label{cor-comm-group-O_s}

Let $f, g$ be slice preserving regular functions on a symmetric slice domain $D$. Then $fg$ is regular on $D$ and $f \star g = g \star f = fg = gf$, that is, $\star$-product coincides with the usual multiplication in $\cO_s(D)$.

\end{corollary}

The above corollary implies that $\cO_s(D)$ is a commutative ring with respect to the usual addition ``$+$" and usual multiplication ``$\cdot$". In addition, the $\star$-product coincides with the usual multiplication in $\cO_s(D)$, and so for the rest of our paper, whenever we deal with only elements in the same ring $\cO_s(D)$, we only use the usual multiplication. On the other hand, suppose that $fg$ is identically to zero on $D$ for some $f, g \in \cO_s(D)$. Thus $f_Ig_I \equiv 0$ on $D_I$ for all $I \in \bS$. Since $f, g \in \cO_s(D)$, $f_I, g_I : D \cap \bC_I \to \bC_I$ are holomorphic functions in $\bC_I$ such that $f_Ig_I \equiv 0$. Since the set of holomorphic functions in $\bC_I$ is an integral domain (see \cite{rem1, rem2}), $f_I \equiv 0$ or $g_I \equiv 0$. By the Identity Principle (see Theorem \ref{thm1.12-gss}), $f \equiv 0$ or $g \equiv 0$, which proves that $\cO_s(D)$ is an integral domain (see Lang \cite{lang} for a notion of integral domain). We record this simple fact in the following.

\begin{proposition}
\label{prop-O_s-commutative-ring}

$\cO_s(D)$ is an integral domain with respect to the usual addition ``$+$" and the usual multiplication ``$\cdot $". Furthermore in $\cO_s(D)$, the multiplication ``$\cdot$" concides with the $\star$-product.

\end{proposition}

Recall from \cite[Proposition 5.26]{GSS} that the set $\cM(D)$ of all semiregular functions on $D$ is a division ring with respect to ``$+$" and the $\star$-product. For each semiregular function $f$ on a symmetric slice domain $D$, we denote by $\cZ(f)$ the zero set of $f$ in $D$, and $\cP(f)$ the set consisting of poles or removable singularities of $f$ in $D$. 

We recall the following results concerned with slice preserving regular functions.

\begin{lemma}
\label{lem-inverse-of-reg-func-in-O_s}
(see \cite[Lemma 5.2]{GSS})

Let $f$ be a slice preserving regular function on a symmetric slice domain $D$ such that $f \not\equiv 0$. Then $\dfrac{1}{f}$ is a slice preserving regular function on $D\setminus \cZ(f)$.

\end{lemma}

\begin{lemma}
\label{lem-zeros-of-slice-preserving-regular-func}
(see \cite[Lemma 3.3]{GSS})

Let $f$ be a slice preserving regular function on a symmetric slice domain $D$. If $f(x + yI) = 0$ for some $x, y \in \bR$ and $I \in \bS$, then $f(x + yJ) = 0$ for all $J \in \bS$, that is, $f$ vanishes identically on $[x + y\bS]$.

\end{lemma}

\begin{lemma}
\label{lem-Zorn-for-zero}

Let $f$ be a slice preserving regular function on a symmetric slice domain $D$ such that $f \not\equiv 0$, and let $p = x + yI$ for some $x, y \in \bR$ and $I \in \bS$. Then
\begin{itemize}

\item [(i)] If $y \ne 0$, then there exists a largest integer $\ell$ and a slice preserving regular function $g : D \to \bH$ such that $g$ has no zeros in $[x + \bS y]$ and 
\begin{align*}
f(q) = ((q - x)^2 + y^2)^{\ell} g(q)
\end{align*}
for all $q \in D$.

\item [(ii)] If $y = 0$,  then there exists a largest integer $\ell$ and a slice preserving regular function $g : D \to \bH$ such that $g(x) \ne 0$ and 
\begin{align*}
f(q) = (q - x)^{\ell} g(q)
\end{align*}
for all $q \in D$.

\end{itemize}

\end{lemma}

\begin{proof}

We first prove part (i). If $f(p) \ne 0$, it follows from Lemma \ref{lem-zeros-of-slice-preserving-regular-func} that $f$ has no zeros in $[x + \bS y]$. Thus one can choose $\ell = 0$, and $h = f$. If $f(p) = 0$, Lemma \ref{lem-zeros-of-slice-preserving-regular-func} implies that $f$ vanishes identically on $[x + \bS y]$. By Lmema \ref{lem-spherical-multiplicity}, $f(q) = ((q - x)^2 + y^2)h(q)$ for some regular function $h: D \to \bH$. 

Assume the contrary to part (i), that is, for every integer $s > 0$, there exists a regular function $g^{(s)} : D \to \bH$ such that 
\begin{align*}
f(q) = ((q - x)^2 + y^2)^sg^{(s)}(q).
\end{align*}
Since $f$ and $((q - x)^2 + y^2)^s$ are slice preserving, $g^{(s)}$ is slice preserving for all integers $s > 0$. Thus the restrictions $f_I, g^{(s)}_I : D \cap \bC_I \to \bC_I$ are holomorphic functions in $\bC_I$ and satisfy
\begin{align*}
f_I(z) = ((z - x)^2 + y^2)^sg^{(s)}_I(z) = (z - (x + Iy))^s(z - (x - Iy))^s g^{(s)}_I(z)
\end{align*}
for all $z \in D \cap \bC_I$ and all integers $s > 0$. So $x + Iy$ is a zero of the holomorphic function $f_I$ of order $s$ for all integers $s > 0$. Thus $f_I \equiv 0$, and therefore by the Identity Principle (see Theorem \ref{thm1.12-gss}), $f$ is identically to zero, a contradiction. Thus there exists a largest positive integer $\ell$ such that 
\begin{align*}
f(q) = ((q - x)^2 + y^2)^{\ell} g(q)
\end{align*}
for some regular function $g : D \to \bH$. The same arguments implies that $g$ is a slice preserving regular function on $D$. Furthermore if $g$ has a zero in $[x + \bS y]$, Lemma \ref{lem-zeros-of-slice-preserving-regular-func} implies that $g$ vanishes identically on $[x + \bS y]$. Therefore $g(q) = ((q - x)^2 + y^2)h(q)$ for some regular function $h$, and thus 
\begin{align*}
f(q) = ((q - x)^2 + y^2)^{\ell + 1} h(q),
\end{align*}
which contradicts the maximality of $\ell$. Thus $g$ has no zeros in $[x + \bS y]$, and part (i) follows immediately.

Using the same arguments as above, part (ii) follows immediately.

\end{proof}

The above lemma still holds for regular functions that are not necesarily slice preserving. The proof of the following result follows the same lines as that of the above lemma. 

\begin{lemma}
\label{lem-Zorn1-for-zero}

Let $f$ be a regular function on a symmetric slice domain $D$ such that $f \not\equiv 0$, and let $p = x + yI$ for some $x, y \in \bR$ and $I \in \bS$. Then
\begin{itemize}

\item [(i)] If $y \ne 0$, then there exists a largest integer $\ell$ and a regular function $g : D \to \bH$ such that $g$ does not vanish identically on $[x + \bS y]$ and 
\begin{align*}
f(q) = ((q - x)^2 + y^2)^{\ell} g(q)
\end{align*}
for all $q \in D$.

\item [(ii)] If $y = 0$,  then there exists a largest integer $\ell$ and a regular function $g : D \to \bH$ such that $g(x) \ne 0$ and 
\begin{align*}
f(q) = (q - x)^{\ell} g(q)
\end{align*}
for all $q \in D$.

\end{itemize}

\end{lemma}

\begin{remark}
\label{rem-quotient-field-of-O_s}

Proposition \ref{prop-O_s-commutative-ring} implies that the quotient field of $\cO_s(D)$ exists (see Lang \cite{lang} for a notion of quotient fields and its construction), and consists of functions of the form $\dfrac{f}{g}$ for $f, g \in \cO_s(D)$ with $g \not\equiv 0$. Throughout the paper, we denote by $\cQ(\cO_s(D))$ the quotient field of $\cO_s(D)$.

Take an arbitrary element $\dfrac{f}{g} \in \cQ(\cO_s(D))$ for some regular functions $f, g \in \cO_s(D)$ with $g \not\equiv 0$.  By Lemma \ref{lem-inverse-of-reg-func-in-O_s}, $\dfrac{1}{g}$ is a slice preserving regular function on $D \setminus \cZ(g)$, and thus $\dfrac{f}{g}$ is a slice preserving regular functionon $D \setminus \cZ(g)$. We claim that every element in $\cZ(g)$ is a pole of $\dfrac{f}{g}$. Indeed, take an arbitrary $p = x + yI$ in $\cZ(g)$ for some $x, y \in \bR$ and $I \in \bS$. It suffices to consider the case when $y \ne 0$. The case when $y = 0$ uses the same arguments. By Lemma \ref{lem-Zorn-for-zero}, there exists a largest positive integer $\ell$ such that
\begin{align*}
g(q) = ((q - x)^2 + y^2)^{\ell} h(q),
\end{align*}
where $h$ is also a slice preserving regular function on $D$ that has no zeros in $[x + \bS y]$. Thus the restrictions $g_I, h_I : D \cap \bC_I \to \bC_I$ of $g$ and $h$ to $D \cap \bC_I$ satisfy
\begin{align*}
\dfrac{1}{g_I(q)} = ((q - x)^2 + y^2)^{-\ell}h_I(q)^{-1},
\end{align*}
and thus
\begin{align*}
\dfrac{f_I(q)}{g_I(q)} = ((q - x)^2 + y^2)^{-\ell}h_I(q)^{-1}f_I(q) = (q - p)^{\ell}(q - \bar{p})^{-\ell}h_I(q)^{-1}f_I(q),
\end{align*}
where $f_I : D \cap \bC_I \to \bC_I$ is the restriction of $f$ to $D \cap \bC_I$.

Since $h_I(p) = h_I(x + Iy) \ne 0$, and $(p - \bar{p})^{-\ell} = (x + Iy - (x - Iy))^{-\ell} = (2yI)^{-\ell} \ne 0$, it follows that $(q - \bar{p})^{-\ell}h_I(q)^{-1}f_I(q)$ is holomorphic in a neiborhood of $p$ in $\bC_I$. Thus $p$ is a pole of $\dfrac{f}{g}$ of order at most $\ell$. Thus every element in $\cZ(g)$ is a pole or a removable singularity of $\dfrac{f}{g}$. Note that since $g$ is a slice preserving regular function on $D$, $\cZ(g)$ consists of spheres $[x + \bS y]$ for some $x, y \in \bR$, and thus it is a symmetric set. Thus $D \setminus \cZ(g)$ is a symmetric slice domain, and therefore $\dfrac{f}{g}$ is a semiregular function on $D$ that is slice preserving regular on $\cD \setminus \cZ(g)$.

\end{remark}

The above remark motivates the following notion of \textbf{slice preserving semiregular functions} on a symmetric slice domain that is a natural generalization of that of slice preserving regular functions.

\begin{definition}
\label{def-slice-preserving-semi-func}

Let $f$ be a semiregular function on a symmetric slice domain $D$. The function $f$ is said to be \textbf{slice preserving} if there exists a symmetric slice domain $A \subset D$ such that the following conditions are true.
\begin{itemize}

\item [(i)] $f$ is a slice preserving regular function on $A$; and

\item [(ii)] every element in the set $P = D \setminus A$ is a pole or a removable singularity of $f$.

\end{itemize}

\end{definition}

We denote by $\cM_s(D)$ the subset of $\cM(D)$ consisting of all slice preserving semiregular functions on $D$. The following result follows immediately from Lemma \ref{lem-commutative-O_s}. 

\begin{lemma}
\label{lem-commutative-M_s}

Let $f, g \in \cM_s(D)$ be slice preserving semiregular functions on a symmetric slice domain $D$. Then $fg$ belongs in $\cM_s(D)$, and
\begin{align*}
f \star g = g \star f = fg = gf.
\end{align*}
Consequently $\cM_s(D)$ is a commutative ring with respect to ``$+$" and the $\star$-regular product.

\end{lemma}

\begin{remark}

We will prove below that $\cM_s(D)$ in fact is a field.

\end{remark}

\begin{proof}

Recall from  \cite[Proposition 5.26]{GSS} that $\cM(D)$ is a division ring with respect to ``$+$" and the $\star$-regular product. Thus it is clear that both $f \star g$ and $g \star f$ are semiregular functions on $D$. On $D \setminus (\cP(f) \cup \cP(g))$, both $f$ and $g$ are regular, and by definition, both $f$ and $g$ are slice preserving regular functions on $D \setminus (\cP(f) \cup \cP(g))$. It thus follows from Lemma \ref{lem-commutative-O_s} that 
\begin{align*}
f \star g = g \star f = fg = gf,
\end{align*}
which proves the lemma.

\end{proof}

 The following result is obvious from Remark \ref{rem-quotient-field-of-O_s}.
\begin{proposition}
\label{prop-quotient-field-contained-in-M_s}

$\cQ(\cO_s(D))$ is a subfield of $\cM_s(D)$.

\end{proposition}

We now prove a structure theorem for $\cM_s(D)$. 

\begin{theorem}
\label{thm-M_s-is-a-field}

Let $D$ be a symmetric slice domain in $\bH$. Then $\cM_s(D)$ is a field with respect to the usual addition ``$+$" and the usual multiplication ``$\cdot$". Furthermore in $\cM_s(D)$, the usual multiplication ``$\cdot$" concides with the $\star$-regular product, and $\cQ(\cO_s(D))$ is a subfield of $\cM_s(D)$, where $\cQ(\cO_s(D))$ is the quotient field of $\cO_s(D)$.

\end{theorem}

\begin{proof}

By Lemma \ref{lem-commutative-M_s} and Proposition \ref{prop-quotient-field-contained-in-M_s}, $\cM_s(D)$ is a commutative ring that contains $\cQ(\cO_s(D))$ as a subfield.

Let $f \not\equiv 0$ be an element in $\cM_s(D)$. Then $f$ is slice preserving regular on $D \setminus \cP(f)$. By  Lemma \ref{lem-inverse-of-reg-func-in-O_s}, $\dfrac{1}{f}$ is a slice preserving regular function on $D \setminus (\cZ(f) \cup \cP(f))$. Using the same arguments as in Remark \ref{rem-quotient-field-of-O_s}, every point in $\cZ(f) \cup \cP(f)$ is a pole or a removable singularity of $\dfrac{1}{f}$, which proves that $\dfrac{1}{f}$ is a slice preserving semiregular function on $D$ Therefore $\cM_s(D)$ is a field.

\end{proof}

\begin{remark}

In Section \ref{sec-weierstrass}, we will prove that for $D = \bH$, 
\begin{align*}
\cQ(\cO_s(\bH)) = \cM_s(\bH),
\end{align*}
that is, the quotient field of the integral domain of all slice preserving entire functions is the ring of all slice preserving semiregular functions on $\bH$. 

In general, we do not know whether $\cQ(\cO_s(D)) = \cM_s(D)$ for every symmetric slice domian $D$ in $\bH$, or whether there exists a symmetric slice domain $D$ in $\bH$ for which $\cQ(\cO_s(\bH))$ is a proper subset of $\cM_s(D)$. If there exists an analogue of the Weierstrass factorization theorem for regular functions on symmetric slice domains in $\bH$, the following conjecture should have a positive answer.

\begin{conjecture}
\label{main-conj}

$\cQ(\cO_s(D)) = \cM_s(D)$ for every symmetric slice domain $D$.

\end{conjecture}

Note that there is an analogue of the Weierstrass factorization theorem for regular functions on the whole quaternions $\bH$ that is due to Gentili and  Vignozzi \cite{gv}. We will recall this theorem in Section \ref{sec-weierstrass}.

\end{remark}

The next result plays a key role in the construction of an important example of spherical divisors on a symmetric slice domain in $\bH$.

\begin{corollary}
\label{cor-to-cor5.27}

Let $f \in \cM_s(D)$ be a slice preserving semiregular function on a symmetric slice domain $D$ such that $f \not\equiv 0$. Let $p = x + yI$ be a point in $D$ for some $x, y \in \bR$ such that $y \ne 0$ and $I \in \bS$. Let $m = \ord_f(p)$ and $n = \ord_f(\bar{p})$. Then the following are true.
\begin{itemize}

\item [(i)] if $\max(n, m) = 0$, then there exists a unique integer $\ell \ge 0$ and a unique slice preserving semiregular function $h$ on $D$ without poles and zeros in $[x + \bS y]$ such that
\begin{align*}
f(q) = ((q - x)^2 + y^2)^{\ell} h(q).
\end{align*}

\item [(ii)] if $\max(n, m) > 0$, then $n = m$, and there exists a unique slice preserving semiregular function $g$ on $D$ without poles and zeros in $[x + \bS y]$ such that
\begin{align*}
f(q) = ((q - x)^2 + y^2)^{-n}  g(q).
\end{align*}

\end{itemize}

\end{corollary}

\begin{proof}

By Theorem \ref{cor5.27-gss},  there exists a unique semiregular function $g$ on $D$ such that the following are true:
\begin{itemize}

\item [(a)] $g$ has no poles in $[x + \bS y]$; and

\item [(b)] $f(q) = ((q - x)^2 + y^2)^{-\max(n, m)}(q - \alpha)^{\star |n - m|}\star g(q)$, where 
\begin{align*}
\alpha =
\begin{cases}
p \; \; &\text{if $n \ge m$}, \\
\bar{p} \; \; &\text{if $m > n$}.
\end{cases}
\end{align*}

\end{itemize}

Suppose that $\max(n, m) = 0$. Then $m = n = 0$, and so $f$ has no poles in $[x + \bS y]$. Thus there exists a symmetric slice domain $U$ that contains $[x + \bS y]$ such that $f$ is regular on $U$, which implies that $f$ is a slice preserving regular function on $U$. By Lemma \ref{lem-Zorn-for-zero}, there exists a largest integer $\ell \ge 0$ and a slice preserving regular function $h : U \to \bH$ that has no zeros in $[x + \bS y]$ such that
\begin{align*}
f(q) = ((q - x)^2 + y^2)^{\ell}h(q)
\end{align*}
for all $q \in U$.

Note that the function $((q- x)^2 + y^2)^{-\ell} f(q)$ is a slice preserving semiregular function on $D$, and identical to $h$ on $U$. By abuse of notation, we also denote this semiregular function on $D$ by $h$. Since $h$ is regular on $U$, and $[x + \bS y] \subset U$, $h$ has no poles or zeros in $[x + \bS y]$, and part (i) of Corollary \ref{cor-to-cor5.27} follows.

Suppose that $\max(n, m) > 0$.We claim that $n = m$. Assume the contrary, i.e., $n \ne m$. Without loss of generality, and in replacing $p$ by $\bar{p}$, if necessary, we can assume that $n > m$. In this case, $\alpha = p$ and $\max(n, m) = n$.  Since $((q - x)^2 + y^2)^n$ is a slice preserving regular function and $g$ has no poles in $[x + \bS y]$, it follows from (a) and (b) above that there exists a symmetric slice domain $U$ that contains $[x + \bS y]$ such that $r(q) = (q - p)^{\star (n - m)}\star g(q)$ is a slice preserving regular function on $U$.

 Since $r$ has a zero at $q = p \in [x + \bS y]$, Lemma \ref{lem-spherical-multiplicity} implies that $r$ vanishes identically on $[x + \bS y]$. By Lemma \ref{lem-zeros-of-product-of-func}, and since the only zero of the regular function $s (q) = (q - p)^{\star (n - m)}$ is $q = p$, we deduce that 
\begin{align}
\label{e0-cor-to-cor5.27}
g(s(q_0)^{-1}q_0 s(q_0)) = 0
\end{align}
for all $q_0 \in [x + \bS y]$ such that $q_0 \ne p$. 

Let $J$ be an element in $\bS$ such that $I, J$ are orthogonal. Set $K = IJ$. By Proposition \ref{prop-inner-cross-prod-relations}, $1, I, J, K$ forms a basis of $\bH$ with the same algebraic properties as the standard basis $1, i, j, k$. Let $q_0 = x + yJ$. By Theorem \ref{cor-1.16-gss},
\begin{align*}
s(q_0) = s(x + yJ) = \dfrac{1}{2}(s(x + yI) + s(x - yI)) + \dfrac{JI}{2}(s(x - yI) - s(x + yI)).
\end{align*}

Since $p = x + yI \in \bC_I$, for all $z \in \bC_I$, $s(z) = (z - p)^{\star (n - m)} = (z - p)^{n - m}$. In particular,
\begin{align*}
s(p) = s(x + yI) = 0
\end{align*}
and
\begin{align}
\label{e1-cor-to-cor5.27}
s(\bar{p}) = s(x - yI) = (x - yI - (x + yI))^{n - m} = (-2y)^{n - m}I^{n - m} \in \bC_I.
\end{align}
Thus
\begin{align*}
s(\bar{p})^{-1}\bar{p}s(\bar{p}) = \bar{p},
\end{align*}
and therefore it follows from (\ref{e0-cor-to-cor5.27}) that
\begin{align}
\label{e1/4-cor-to-cor5.27}
g(\bar{p}) = 0.
\end{align}

We see that
\begin{align*}
s(q_0) = \dfrac{s(x - yI)}{2}(1 + JI) = 2^{n - m - 1}(-y)^{n - m}I^{n - m}(1 + JI).
\end{align*}
Note that $1 + JI = 1 - IJ = 1 - K$, and thus
\begin{align*}
(1 + JI)^{-1} = (1 - K)^{-1} = \dfrac{1 + K}{2}.
\end{align*}
Thus
\begin{align*}
s(q_0)^{-1}q_0 s(q_0) &= x + s(q_0)^{-1}Js(q_0)y \\
&= x + (1 + JI)^{-1}I^{m - n}JI^{n - m}(1 + JI)y \\
&= x + \dfrac{1}{2}(1 + K) I^{m - n} J I^{n - m}(1 - K)y
\end{align*}

We know that for every $a \in \bZ$,
\begin{align}
\label{e1/2-cor-to-cor5.27}
I^a =
\begin{cases}
1 \; &\text{if $a \equiv 0 \pmod{4}$,}\\
I \; &\text{if $a \equiv 1 \pmod{4}$,} \\
-1 \; &\text{if $a \equiv 2 \pmod{4}$,} \\
-I \; &\text{if $a \equiv -1 \pmod{4}$}.
\end{cases}
\end{align}

Since $J(1 \pm K) = (1 \mp K)J$ and $I(1 \pm K) = (1 \mp K)I$, we deduce that 
\begin{align*}
\dfrac{1}{2}(1 + K) I^{m - n} J I^{n - m}(1 - K) =
\begin{cases}
-I \; &\text{if $n - m\equiv 0, 2 \pmod{4}$,}\\
 I \; &\text{if $n - m \equiv \pm 1 \pmod{4}$.} \\
\end{cases}
\end{align*}
Thus
\begin{align}
\label{e2-cor-to-cor5.27}
s(q_0)^{-1}q_0 s(q_0) = 
\begin{cases}
x + Iy = p \; &\text{if $n - m \equiv \pm 1 \pmod{4}$,} \\
x - Iy = \bar{p} \; &\text{if $n - m \equiv 0, 2 \pmod{4}$.} 
\end{cases}
\end{align}

We consider the following two cases.

\textit{Case 1. $n - m \equiv \pm 1 \pmod{4}$.}

In this case, it follows from , (\ref{e0-cor-to-cor5.27}), (\ref{e2-cor-to-cor5.27}) that 
\begin{align*}
g(p) = g(x + Iy) = 0.
\end{align*}
By (\ref{e1/4-cor-to-cor5.27}), $g(\bar{p}) = g(x - Iy) = 0$. Thus the Extension Formula (see  Theorem \ref{cor-1.16-gss}) that $g$ vanishes identically on $[x + \bS y]$. By Lemma \ref{lem-Zorn1-for-zero}, there exists a largest positive integer $\ell$ such that $g(q) = ((q - x)^2 + y^2)^{\ell}g_1(q)$ for some regular function $g_1 : U \to \bH$ such that $g_1$ does not vanish identically on $[x + \bS y]$. Thus
\begin{align*}
f(q) = ((q - x)^2 + y^2)^{\ell - n} (q - p)^{\star (n - m)} \star g_1(q).
\end{align*}
Repeating the same arguments as above with $g_1$ in the role of $g$, we deduce that $g_1$ vanishes identically on $[x + \bS y]$, which is a contradiction.

\textit{Case 2. $n - m \equiv 0, 2 \pmod{4}$.}

By (\ref{e1/4-cor-to-cor5.27}), $g(\bar{p}) = g(x - Iy) = 0$. By Lemma \ref{lem-spherical-multiplicity}, there exists a regular function $h : U \to \bH$ such that
\begin{align*}
g(q) = (q - \bar{p})\star h(q)
\end{align*}
for all $q \in U$. Since $(q - p) \star (q - \bar{p}) = ((q - x)^2 + y^2)$, we deduce that
\begin{align}
\label{e3-cor-to-cor5.27}
f(q) &= ((q - x)^2 + y^2)^{\ell - n} (q - p)^{\star (n - m)} \star g(q) \nonumber \\
&= ((q - x)^2 + y^2)^{\ell - n} (q - p)^{\star (n - m)} \star (q - \bar{p}) \star h(q)  \nonumber \\
 &= ((q - x)^2 + y^2)^{\ell - n + 1} (q - p)^{\star (n - m - 1)} \star h(q).
\end{align}

Since $n > m$ and $n - m \equiv 0, 2 \pmod{4}$, we deduce that $n - m - 1 > 0$. Since $n - m - 1 \equiv \pm 1 \pmod{4}$, we use the same arguments and (\ref{e3-cor-to-cor5.27}) as in \textit{Case 1} with $(q - p)^{\star (n - m - 1)} \star h(q)$ in the role of $(q - p)^{\star (n  - m )}\star g(q)$, we deduce that 
\begin{align*}
h(s(q_0)^{-1}q_0 s(q_0)) = 0
\end{align*}
for all $q_0 \in [x + \bS y]$ such that $q_0 \ne p$, and the same arguments imply that $h$ vanishes identically on $[x + \bS y]$, which reaches the same contradiction as in \textit{Case 1}.

From \textit{Cases 1 and 2}, we deduce that $n = m$, and thus $p = x + Iy$ is a pole of $f$ of order $n > 0$. On the other hand, we have
\begin{align}
\label{e4-cor-to-cor5.27}
f(q) = ((q - x)^2 + y^2)^{-n} g(q),
\end{align}
where $g$ has no poles in $[x + \bS y]$. We contend that $g$ has no zeros in $[x + \bS y]$. Assume the contrary, i.e., $g$ has a zero in $[x + \bS y]$. Since $f$ and $((q - x)^2 + y^2)^{-n}$ are slice preserving, and $g$ has no poles in $[x + \bS y]$, there exists a symmetric slice domain $U$ that contains $[x + \bS y]$ such that $g$ is a slice preserving regular function on $U$. Since $g$ has a zero in $[x + \bS y]$, it follows from Lemma \ref{lem-zeros-of-slice-preserving-regular-func} that $g$ vanishes identically on $[x + \bS y]$, and it thus follows from Lemma \ref{lem-Zorn-for-zero} there exists a largest positive integer $\ell$ such that $g(q) = ((q - x)^2 + y^2)^{\ell} h(q)$, where $h$ is a regular function on $U$ that has no zeros in $[x + \bS y]$. 

 By (\ref{e4-cor-to-cor5.27}), we see that
\begin{align*}
f(q) = ((q - x)^2 + y^2)^{-n + \ell} h(q),
\end{align*}
and thus the restrictions $f_I, h_I$ of $f, h$ to $U \cap \bC_I$ satisfy 
\begin{align*}
f_I(q) = (q - p)^{-n + \ell}(q - \bar{p})^{-n + \ell} h_I(q).
\end{align*}
Since $h_I(p) = h_I(x + Iy) \ne 0$ and $h_I$ is a holomorphic function on $U \cap \bC_I$, we deduce from the above equation that the order of the pole $p = x + Iy$ of $f_I$ is precisely $|n - \ell|$, which is strictly less than $n$, a contradiction. Thus $g$ has no zeros in $[x + \bS y]$, and the corollary follows immediately.

\end{proof}

Using the same arguments as Corollary \ref{cor-to-cor5.27}, we obtain the following.

\begin{corollary}
\label{cor-to-cor5.27-1}

Let $f \in \cM_s(D)$ be a slice preserving semiregular function on a symmetric slice domain $D$ such that $f \not\equiv 0$. Let $x$ be a point in $D \cap \bR$. Let $n = \ord_f(x)$. Then the following are true.
\begin{itemize}

\item [(i)] if $n = 0$, then there exists a unique integer $\ell \ge 0$ and a unique slice preserving semiregular function $h$ on $D$ without poles and zeros at $x$ such that
\begin{align*}
f(q) = (q - x)^{\ell} h(q).
\end{align*}

\item [(ii)] if $n > 0$, then there exists a unique slice preserving semiregular function $g$ on $D$ without poles and zeros at $x$ such that
\begin{align*}
f(q) = (q - x)^{-n}  g(q).
\end{align*}

\end{itemize}

\end{corollary}

Corollaries \ref{cor-to-cor5.27} and \ref{cor-to-cor5.27-1} motivates the following notion of \textbf{spherical order} of a slice preserving semiregular function at a point in its domain. This notion will play a key role in a notion of spherical divisor that we will introduce later.

\begin{definition}
\label{def-spherical-order}

Let $f$ be a slice preserving semiregular function on a symmetric slice domain $D$ such that $f \not\equiv 0$, and let $p = x + Iy \in D$ for some $x, y \in \bR$ and $I \in \bS$. If $y \ne 0$, the \textbf{spherical order of $f$ at $p$} is the unique integer $\ell$ such that
\begin{align*}
f(q) = ((q - x)^2 + y^2)^{\ell} g(q)
\end{align*}
for some slice preserving semiregular function on $D$ that have neither poles nor zeros in $[x + \bS y]$. 

If $y = 0$, the \textbf{spherical order of $f$ at $p = x \in \bR$} is the unique integer $\ell$ such that
\begin{align*}
f(q) = (q - x)^{\ell} g(q)
\end{align*}
for some slice preserving semiregular function on $D$ that have neither poles nor zeros at $p = x$.

We denote by $\sord_f(p)$ the spherical order of $f$ at $p$.

By Corollaries \ref{cor-to-cor5.27} and \ref{cor-to-cor5.27-1}, such a unique integer $\sord_f(p)$ exists.

\end{definition}

Letting $p$ range over $D$, one obtains a map $\sord_f : D \to \bZ$ that sends each $p \in D$ to a unique integer $\sord_f(p)$.

The next results follow immediately from Corollaries \ref{cor-to-cor5.27} and \ref{cor-to-cor5.27-1}.

\begin{proposition}
\label{prop-eqn-between-sord-and-ord}

Let $f \in \cM_s(D)$ for some symmetric slice domain $D$ in $\bH$ such that $f \not\equiv 0$, and let $p = x +Iy$ be an element in $D$ for some $x, y \in \bR$ and $I \in \bS$. If $\max(\ord_f(p), \ord_f(\bar{p})) > 0$, then the following are true.
\begin{itemize}

\item[(i)] $\ord_f(p) = \ord_f(\bar{p})$.

\item [(ii)] every point in $[x + \bS y]$ is a pole of $f$ of order $\ord_f(p)$.

\item [(iii)] $\sord_f(p) = -\ord_f(p)$.

\end{itemize}

\begin{remark}

Note that we allow $y = 0$ in the above proposition. In that case, $[x + \bS y]$ is the singleton set $\{x\}$, and $\bar{p} = p$.

\end{remark}

\end{proposition}

The following result follows immediately from Corollaries \ref{cor-to-cor5.27} and \ref{cor-to-cor5.27-1}.

\begin{proposition}
\label{prop-constant-sord-to-each-sphere}

Let $f \in \cM_s(D)$ for some symmetric slice domain $D$ such that $f \not\equiv 0$. For each $[x + \bS y] \subset D$ with $x, y \in \bR$, the restriction of $\sord_f$ to $[x + \bS y]$ is constant, i.e., $\sord_f(x + Iy) = \sord_f(x + Jy)$ for all $I, J \in \bS$.

\end{proposition}

\begin{lemma}
\label{l-main-lemma1}

Let $f$ be a slice preserving semiregular function in $\cM_s(D)$ for some symmetric slice domain $D$ such that $f \not\equiv 0$. Write
\begin{align*}
D = \cup_{i \in \cI} \bS_i,
\end{align*}
where for each $i \in \cI$, $\bS_i = x_i + y_i\bS$ for some $x_i, y_i \in \bR$ such that $x_i + y_iI_i \in D$ for some $I_i \in \bS$ and $\bS_i \cap \bS_j = \emptyset$. For every $q \in D$ with $\sord_f(q) \ne 0$, there exists a symmetric slice neighborhood of $q$, say $U_q \subset D$ such that there are only finitely many spheres $\bS_j$ for which the restriction of $\sord_f$ to $U_q \cap \bS_j$ is nonzero.

\end{lemma}

\begin{proof}

Take an arbitrary $q_0 \in D$. Assume the contrary, i.e, for any slice symmetric neighborhood $U$ of $q_0$, there exist infinitely many $\bS_j$ such that the restriction of $\sord_f$ to $U \cap \bS_j$ is nonzero. Take a symmetric compact subset $K$ of $D$ that contains $q_0$. We know from Proposition \ref{prop-constant-sord-to-each-sphere} that the restriction to $\sord_f$ to $\bS_i$ is constant for every $i \in \cI$. Thus there exist infinitely many spheres, say $(\bS_{n})_{n \in A}$, where $A$ is an infinite subset of $\bZ$ such that $\sord_f(q) \ne 0$ for any $q \in K \cap \bS_{n}$ for any $n \in A$. Since the restriction of $\sord_f$ to $K \cap \bS_{n}$ is nonzero, $K \cap \bS_{n}$ is nonempty. Thus $\bS_{n} \subset K$ for every $n \in A$ since $K$ is symmetric.

Set
\begin{align*}
B = \{n \in A \; |\; y_{n} = 0\},
\end{align*}
and
\begin{align*}
C = \{n \in A \; | \;  y_n \ne 0\}.
\end{align*}

Fix an element $I \in \bS$. By Corollaries \ref{cor-to-cor5.27} and \ref{cor-to-cor5.27-1}, for each $n \in B$, there exists a unique slice preserving semiregular function $\beta_{n}$ on $D$ that have neither zeros nor poles at $x_n$ such that
\begin{align}
\label{e1-main-lem1}
f(q) = (q - x_{n})^{\sord(x_{n})} \beta_{n}(q).
\end{align}

Similarly, for each $n \in C$, there exists a unique slice preserving semiregular function $\alpha_{n}$ on $D$ that have neither zeros nor poles in $[x_{n} + \bS y_{n}]$  such that
\begin{align}
\label{e1-main-lem2}
f(q) &= ((q - x_{n})^2 + y_{n}^2)^{\sord_f([x_{n} + y_{n}\bS])} \alpha_{n}(q) \nonumber \\
&= (q - (x_{n} + y_{n}I))^{\sord_f([x_{n} + y_{n}\bS])}(q - (x_{n} - y_{n}I))^{\sord_f([x_{n} + y_{n}\bS])}\alpha_{n}(q).
\end{align}

$\star$ \textit{Case 1. There exist infinitely many spheres, say $(\bS_{n})_{n \in \cJ}$, where $\cJ$ is an infinite subset of $A$ such that $\sord_f(q) > 0$ for any $q \in K \cap \bS_{n}$ for any $n \in \cJ$. }

Since $A = B \cup C$, it follows that either $\cJ \cap B$ or $\cJ \cap C$ is infinite. 

Suppose first that $\cJ \cap B$ is infinite. By (\ref{e1-main-lem1}), every real number $\{x_n\}_{n \in \cJ \cap B}$ is a zero of $f$. Since $x_n \in K$ and $x_n \in \bR \subset \bC_I$ for every $n \in \cJ \cap B$ and $K \cap \bC_I$ is a compact subset of $D \cap \bC_I$, there exists a subsequence $\{x_{i_n}\}_{n \in \cL}$ for some infinite subset $\cL$ of $\cJ \cap B$ such that the subsequence $\{x_{i_n}\}_{n \in \cL}$ converges to a limit point in $K \cap \bC_I$. Thus the zero set of $f$ contains a convergent sequence of elements in $K \cap \bC_I$, and therefore by the Identity Theorem (see Theorem \ref{thm1.12-gss}), $f \equiv 0$, which is a contradiction. 

Suppose now that $\cJ \cap C$ is infinite. From (\ref{e1-main-lem2}), we deduce that every term in the sequence $\{x_{n} + y_{n}I\}_{n \in \cJ \cap C}$ are zeros of $f$. Since $K \cap \bC_I$ is a compact subset of $D \cap \bC_I$, and $x_{n} + y_{n}I \in K \cap \bC_I$ for every $n \in \cJ \cap \bC_I$, there exists a subsequence $\{x_{i_n} + y_{i_n}I\}_{n \in \cL}$ for some infinite subset $\cL$ of $\cJ \cap \bC_I$ such that $\{x_{i_n} + y_{i_n}I\}_{n \in \cL}$ converges to a limit point in $K \cap \bC_I$. Thus the zero set of $f$ contains a convergent sequence of elements in $K \cap \bC_I$, and therefore by the Identity Theorem (see Theorem \ref{thm1.12-gss}), $f \equiv 0$, which is a contradiction.

$\star$ \textit{Case 2. There exist infinitely many spheres, say $(\bS_{n})_{n \in \cH}$, where $\cH$ is an infinite subset of $A$ such that $\sord_f(q) < 0$ for any $q \in K \cap \bS_{n}$ for any $n \in \cH$ }

Set, for each $n \in \cH$,
\begin{align*}
q_{n} = x_{n} + y_{n}I \in K \cap \bC_I.
\end{align*}
By Proposition \ref{prop-eqn-between-sord-and-ord}, $\ord_f(q_n) = - \sord_f(q_n) > 0$, and thus $q_n$ is a pole of $f$ of order $\ord_f(q_n)$ for every $n \in \cH$. Thus every term in the sequence $\{q_n\}_{n \in \cH}$ is a pole of the meromorphic function $f_I : D \cap \bC_I \to \bC_I$, where $f_I$ is the restriction of $f$ to $D \cap \bC_I$. Since $K \cap \bC_I$ is a compact subset of $D \cap \bC_I$, there exists a convergent subsequence, say $\{q_n\}_{n \in \cL}$ in $K \cap \bC_I$ for some infinite subset $\cL$ of $\cH$ whose limit point also belongs in $K \cap \bC_I$. Since $f_I : D\cap \bC_I \to \bC_I$ is a meromorphic function on $D \cap \bC_I$, the set of its poles in $D \cap \bC_I$ is discrete, which is a contradiction to the fact that the convergent sequence $\{q_n\}_{n \in \cL}$ belongs in $K \cap \bC_I$. 

From \textit{Cases 1 and 2}, the lemma follows immediately.

\end{proof}

For a slice preserving semiregular function $f$ on a symmetric slice domain $D$, we denote by $\divi(f)$ the map from $D$ to $\bZ$ which sends each element $q \in D$ to $\sord_f(q)$, i.e., $\divi(f)(q) = \sord_f(q)$. Such divisors are called \textbf{principal spherical divisors} on $D$. The next result is immediate from the above lemma and Proposition \ref{prop-constant-sord-to-each-sphere}.

\begin{corollary}
\label{cor-main-c1}

For every slice preserving semi regular function $f$ on a symmetric slice domain $D$ in $\bH$, the map $\divi(f) : D \to \bZ$ is a spherical divisor on $D$.

\end{corollary}

A spherical divisor $\fd$ on a symmetric slice domain $D$ is called \textit{positive} and written $\fd \ge 0$ if $\fd(z) \ge 0$ for all $z \in D$. The following result is obvious.

\begin{lemma}

For every slice preserving regular function $f$ on a symmetric slice domain $D$, the principal spherical divisor $\divi(f)$ is positive.

\end{lemma}


Recall from Theorem \ref{thm-M_s-is-a-field} that $\cM_s(D)$ is a field with respect to the usual addition and multiplication ``$+$'' and ``$\cdot$".  Let $\cM_s(D)^{\times}$ denote its abelian multiplicative group. We introduce a map $\Gamma : \cM_s(D)^{\times} \to \Div(D)$ which is defined by sending each slice preserving semiregular function $f \in \cM_s(D)$ to the principal spherical divisor $\divi(f) \in \Div(D)$. The following is immediate from Corollaries \ref{cor-to-cor5.27} and \ref{cor-to-cor5.27-1}.

\begin{lemma}
\label{lem-holomorphy-criterion}

\begin{itemize}

\item[]
\item [(i)] $\Gamma$ is a group homomorphism.

\item [(ii)] $f \in \cM_s(D)^{\times}$ is regular in $D$ if and only if $\divi(f) \ge 0$;

\item [(iii)] $f \in \cM_s(D)^{\times}$ is a unit in $\cO_s(D)$ if and only if $\divi(f) = 0$. 

\end{itemize}

\end{lemma}

\section{Weierstrass factorization theorem}
\label{sec-weierstrass}

In this section, we will prove that the quotient field $\cQ(\cO_s(\bH))$ of $\cO_s(\bH)$ is $\cM_s(\bH)$, which provides an evidence to Conjecture \ref{main-conj}. Our proof will be based on a quaternionic analogue of the Weierstrass factorization theorem for entire functions on $\bH$ that is due to Gentili and Vignozzi \cite{gv} (see also \cite[Theorem 4.34]{GSS}). In fact, we only need a special case of the theorem of Gentili and Vignozzi's for slice preserving entire functions on $\bH$.

\begin{lemma}
\label{lem-nth-roots-of-exp}

Let $P(q) = a_nq^n + a_{n -1}q^{n - 1} + \cdots+ a_1q + a_0 $ be a polynomial of degree $n$ in variable $q$ over $\bH$, where the $a_i$ are in $\bR$ such that $a_n \ne 0$. Then
\begin{itemize}

\item [(i)] the function $e^{P(q)} := \exp(P(q))$ is a slice preserving regular function in $\cO_s(\bH)$.

\item [(ii)] for every positive integer $\ell$, there exists a slice preserving regular function $Q_{\ell} \in \cO_s(\bH)$ such that
\begin{align*}
Q_{\ell}(q)^{\ell}= e^{P(q)},
\end{align*}
that is, $Q_{\ell}$ is an $\ell$-th root of $e^{P(q)}$.

\end{itemize}

\end{lemma}
\begin{proof}

We first verify part (i). Since the $a_i$ are in $\bR$, $P$ is a slice preserving regular function in $\cO_s(\bH)$. Recall that the exponential function $exp(q)$ is defined as
\begin{align*}
\exp(q) = e^q = \sum_{n \ge 0} \dfrac{q^n}{n!},
\end{align*}
which proves that $\exp(q)$ is a slice preserving regular function in $\cO_s(\bH)$ (see \cite[Section 4.2]{GSS} for the notion of $\exp(q)$). Thus by Lemma \ref{lem-comp-of-slice-reg-func}, the function $\exp(P(q)) = \exp \circ P (q)$ is a slice preserving regular function in $\cO_s(\bH)$, which proves part (i).

For part (ii), take an arbitrary positive integer $\ell$. Set
\begin{align*}
h_{\ell}(q) =\dfrac{ a_n}{\ell}q^n +\dfrac{a_{n - 1}}{\ell}q^{n - 1} + \cdots + \dfrac{ a_1}{\ell}q + \dfrac{a_0}{\ell}.
\end{align*}
We see that $h_{\ell}(q) =\dfrac{ P(q)}{\ell}$, which implies that $h_{\ell} \in \cO_s(\bH)$. It thus follows from Lemma \ref{lem-comp-of-slice-reg-func} that the function $Q_{\ell}(q) = \exp(h_{\ell}(q))$ is a slice preserving regular function in $\cO_s(\bH)$. We see that
\begin{align*}
Q_{\ell}(q)^{\ell} = \exp(h_{\ell}(q))^{\ell} = \exp(P(q)),
\end{align*}
which proves part (ii).

\end{proof}

We recall the following analogue of Weierstrass factors for the quaternions $\bH$.

\begin{definition}
\label{def-regular-factor}

For a given $a \in \bR \setminus \{0\}$, and $n \in \bZ_{\ge 1}$, let $g : \bH \to \bH$ be the function defined by
\begin{align*}
g_{n, a}(q) = e^{qa^{-1} + \dfrac{1}{2}q^2a^{-2} + \cdots + \dfrac{1}{n}q^n a^{-n}}.
\end{align*}
The function $g_{n, a}$ is called the \textbf{convergence-producing regular factor associated to $n$ and $a$}. For $n = 0$, simply set $g_{0, a}(q) = 1$.

\end{definition}

\begin{remark}

\begin{itemize}

\item []

\item [(i)] Since the polynomial $qa^{-1} + \dfrac{1}{2}q^2a^{-2} + \cdots + \dfrac{1}{n}q^n a^{-n}$ has real coefficients, it follows from Lemma \ref{lem-nth-roots-of-exp} that $g_{n, a}$ is a slice preserving regular function, and thus $g_{n, a} \in \cO_s(\bH)$.

\item [(ii)] It is known (see \cite{gv, GSS}) that $g_{n, a}$ has no zeros. 

\end{itemize}

\end{remark}

\begin{lemma}
\label{Weierstrass-lem}

Let $(a_n)_{n \in \bZ_{\ge 0}}$ be a sequence of nonzero real numbers such that $\lim_{n \to \infty} |a_n| = \infty$. For each $n \in \bZ_{\ge 0}$, let $g_{n, a_n}$ be the convergence-producing regular factor associated to $n$ and $a_n$. Then
\begin{align*}
\cP_{(a_n)_{n \in \bZ_{\ge 0}}}(q) = \prod^{\star}_{n \in \bZ_{\ge 0}}(1 - qa_n^{-1})\star g_{n, a_n}(q)  = \prod_{n \in \bZ_{\ge 0}} (1 - qa_n^{-1})g_{n, a_n}(q)
\end{align*}
converges compactly in $\bH$, and defines a slice preserving regular function in $\cO_s(\bH)$. Moreover, for every $n \in \bZ_{\ge 0}$,  the zero set of $\cP_{(a_n)_{n \in \bZ_{\ge 0}}}(q)$ is exactly the sequence $(a_n)_{n \ge \bZ_{\ge 0}}$.

\end{lemma}

\begin{proof}

Clear from \cite[Theorem 4.28]{GSS}.

\end{proof}

\begin{remark}
\label{rem-W-lemma}

Note that in the definition of $\cP_{(a_n)_{n \in \bZ_{\ge 0}}}(q)$, the factors $(1 - qa_n^{-1})g_{n, a_n}(q)$ commute with each other since the functions $(1 - qa_n^{-1})g_{n, a_n}(q)$ are slice preserving regular.

\end{remark}

By Lemma \ref{lem-zeros-of-slice-preserving-regular-func}, every slice preserving entire function has zeros as real numbers of spherical zeros of the form $[x + \bS y]$ for some $x, y \in \bR$.

\begin{theorem}
\label{quaternionic-Weierstrass-thm}
(See \cite{gv} and \cite[Theorem 4.34]{GSS})

Let $f : \bH \to \bH$ be a slice preserving entire function, i.e., $h \in \cO_s(\bH)$. Suppose that $m \in \bZ_{\ge 0}$ is the multiplicity of $f$ at $0$, $\{b_n\}_{n \in \bZ_{\ge 0}} \subset \bR \setminus \{0\}$ is the sequence of the other real zeros of $f$, $\{\bS_n\}_{n \in \bZ_{\ge 0}}$ is the sequence of the spherical zeros of $f$. If all the zeros listed above are repeated according to their multiplicities, then there exists a nowhere vanishing slice preserving entire function $h$, and for all $n \in \bZ_{\ge 0}$, there exist elements $c_n \in \bS_n$ such that
\begin{align*}
f(q) = q^m \star \cR(q) \star  \cS(q) \star h(q) = q^m\cR(q)\cS(q) h(q),
\end{align*}
where
\begin{align*}
\cR(q) = \prod_{n = 0}^{\infty} (1 - qb_n^{-1}) e^{qb_n^{-1} + \dfrac{1}{2}q^2b_n^{-2} + \cdots + \dfrac{1}{n}q^n b_n^{-n}},
\end{align*}
and
\begin{align*}
\cS(q) = \prod_{n = 0}^{\infty}\left(\dfrac{q^2}{|c_n|^2} - \dfrac{2q\Re(c_n)}{|c_n|^2} + 1\right)e^{q\frac{2\Re(c_n)}{|c_n|^2}  + \frac{1}{2}q^2\frac{2\Re(c_n^2)}{|c_n|^4}  + \cdots + \frac{1}{n}q^n \frac{2\Re(c_n^n)}{|c_n|^{2n}}}.
\end{align*}

\end{theorem}

\begin{proof}

By Lemma \ref{lem-zeros-of-slice-preserving-regular-func}, the zero set of $f$ only consist of real numbers or contain spheres of the form $[x + \bS y]$ for some $x, y \in \bR$ with $y \ne 0$ as its subsets. Thus the theorem is an immediate conquesence of \cite[Theorem 4.34]{GSS}.

\end{proof}

Theorem \ref{quaternionic-Weierstrass-thm} has an important corrollary whose proof is clear.

\begin{theorem}
\label{existence-thm}

Let $\fd : \bH \to \bZ$ be a spherical divisor on $\bH$ such that $\fd \ge 0$. Then $\fd$ is a principal spherical divisor on $\bH$, i.e., there exists a slice preserving regular entire function $f$ in $\bH$ such that $f \not\equiv 0$ and $\fd = \divi(f)$.

\end{theorem}

We now prove the main theorem in this section.

\begin{theorem}
\label{thm-M_s(H)-quotient-field}

The field $\cM_s(\bH)$ of slice preserving semiregular functions in $\bH$ is the quotient field of the integral domain $\cO_s(\bH)$ of slice preserving regular functions in $\bH$. Equivalently, for every slice preserving semiregular functions $h$ in $\bH$, there exist two slice preserving regular functions $f$ and $g$ without common zeros in $\bH$ such that $h = \dfrac{f}{g}$.

\end{theorem}

\begin{proof}

The case $h \equiv 0$ is obvious since one can choose $f \equiv 0$ and $g \equiv 1$.

Suppose that $h \not\equiv 0$. Let $\fd^{+} : \bH \to \bZ$ and $\fd^{-} : \bH \to \bZ$ be spherical divisors on $\bH$ defined by
\begin{align*}
\fd^{+}(q) &= \max(0, \sord_h(q)), \\
\fd^{-}(q) &= \max(0, -\sord_h(q))
\end{align*}
for every $q \in \bH$. It is clear that $\divi(h) = \fd^{+} - \fd^{-}$.

Since $\fd^{-} \ge 0$, it follows from Theorem \ref{existence-thm} that there exists a slice preserving entire function $g$ in $\bH$ such that $g \not\equiv 0$ and $\fd^{-} = \divi(g)$. Set 
\begin{align*}
f = gh \not\equiv 0.
\end{align*}
It is clear that $f$ is a slice preserving semiregular function in $\bH$. By Lemma \ref{lem-holomorphy-criterion}, $\divi(f) = \divi(g) + \divi(h) = \fd^{+} \ge 0$, and thus $f$ is a slice preserving regular function in $\bH$.

We now prove that $\cZ(f) \cap \cZ(g)$ is empty. Assume the contrary, i.e., there exists $q_0 \in \cZ(f) \cap \cZ(g)$. Then $\sord_f(q_0) = \fd^{+}(q_0) = \max(0, \sord_h(q_0)) > 0$, which implies that $\sord_h(q_0) > 0$. On the other hand,  $\sord_g(q_0) = \fd^{-}(q_0) = \max(0, -\sord_h(q_0)) > 0$, which implies that $\sord_h(q_0) < 0$, a contradiction. Thus $\cZ(f) \cap \cZ(g)$ is empty, and the theorem follows immediately.

\end{proof}

\section{Bers's theorem}
\label{sec-bers}

In this section, we prove a quaternionic analogue of Bers's theorem. For the classical Bers theorem for rings of analytic functions, see Bers \cite{bers} or Remmert \cite{rem2}.

Let $D$ be a symmetric slice domain in $\bH$. One can equip the integral domain $\cO_s(D)$ of slice preserving slice regular functions in $D$ with the natural action both on left and right by the reals $\bR$. With this action, $\cO_s(D)$ becomes a (commutative) algebra over $\bR$ (see Lang \cite{lang} for this notion). The quaternions $\bH$ is also equipped with the same action both on left and right by the reals $\bR$, and it is a division algebra over $\bR$. Note that every real number commutes with elements in $\cO_s(D)$ as well as in $\bH$.

\begin{definition}
\label{def-char}
(characters of $\cO_s(D)$)

A map $\chi$ from $\cO_s(D)$ to $\bH$ is called a character of $\cO_s(D)$ if the following are satisfied:
\begin{itemize}

\item [(i)] $\chi(f + g) = \chi(f) + \chi(g)$ for any $f, g \in \cO_s(D)$;

\item [(ii)] $\chi(rf) = \chi(fr) = r\chi(f) = \chi(f)r$ for any $r \in \bR$ and $f \in \cO_s(D)$; and

\item [(iii)] $\chi(fg) = \chi(f)\chi(g)$ for any $f, g \in \cO_s(D)$.

\end{itemize}

In brevity, a character $\chi : \cO_s(D) \to \bR$ is simply an $\bR$-algebra homomorphism from $\cO_s(D)$ to $\bH$.

\end{definition}

\begin{remark}
\label{rem-nontrivial-char}

\begin{itemize}

\item [(i)] For the rest of the paper, let $\b1 : D \to \bH$ be the slice preserving regular function on $D$ defined by $\b1(q) = 1$ for every $q \in D$. Note that $\b1^2 = \b1$. From part (iii) above, we deduce that $\chi(\b1) = \chi(\b1^2) = \chi(\b1)^2$, and thus either $\chi(\b1) = 0$ or $\chi(\b1) = 1$.

Suppose that $\chi(\b1) = 0$. For any $f \in \cO_s(D)$, we see that $\chi(f) = \chi(f \b1) = \chi(f)\chi(\b1) = 0$, and thus $\chi \equiv 0$. In this case we say that $\chi$ is a trivial character. If $\chi(\b1) = 1$, $\chi$ is called a nontrivial character of $\cO_s(D)$. It is obvious that $\chi(c) = c$ for any $c \in \bR$ if $\chi$ is nontrivial.

\item [(ii)] From (iii) in the above definition, and since $\cO_s(D)$ is a commutative ring, $\chi(fg) = \chi(gf) = \chi(f)\chi(g) = \chi(g)\chi(f)$ for all $f, g \in \cO_s(D)$. 

\item [(iii)] It is clear from (i) in the above definition that $\chi(0) = 0$.

\end{itemize}

\end{remark}

\begin{lemma}
\label{lem-image-of-char-Bers}

Let $\chi : \cO_s(D) \to \bH$ be a nontrivial character of $\cO_s(D)$. Then there exists an element $I \in \bS$ such that $\chi(f) \in \bC_I$ for every $f \in \cO_s(D)$.

\end{lemma}

\begin{proof}

If $\chi(f) \in \bR$ for every $f \in \cO_s(D)$, then every element $I \in \bS$ satisfies the conclusion of the lemma. 

Suppose that there exists a slice preserving regular function $f_0 \in \cO_s(D)$ such that $c_0 = \chi(f_0) \in \bH \setminus \bR$. One can write $c_0$ in the form $c_0 = a + Ib$ for some $I \in \bS$ and $a, b \in \bR$ with $b\ne 0$. From Remark \ref{rem-nontrivial-char}(ii), for any $f \in \cO_s(D)$,
\begin{align*}
c_0\chi(f) = \chi(f_0)\chi(f) = \chi(f)\chi(f_0) = \chi(f) c_0.
\end{align*}

Fix an arbitrary $f \in \cO_s(D)$, and let $\chi(f) = \alpha + J \beta$ for some $\alpha, \beta \in \bR$ and $J \in \bS$. From the above equation,
\begin{align*}
(a + Ib)(\alpha + J\beta) = (\alpha + J\beta)(a + Ib),
\end{align*}
and thus
\begin{align*}
\beta b IJ = b\beta JI.
\end{align*}

If $\beta = 0$, it follows that $\chi(f) = \alpha \in \bR \subset \bC_I$. If $\beta \ne 0$, it follows that $IJ = JI$. 

Let $\langle I, J \rangle \in \bR$ be the Euclidean inner product of $I$ and $J$ as vectors in $\bR^4$, and $I \times J $ be the natural cross product of $I, J$ in $\bR^3$. It is known from Proposition \ref{prop-inner-cross-prod-relations} that 
\begin{align*}
IJ = -\langle I, J\rangle + I \times J.
\end{align*}
Similarly, 
\begin{align*}
JI = -\langle J, I\rangle + J \times I,
\end{align*}
and thus
\begin{align*}
I \times J = J \times I = - I \times J,
\end{align*}
which implies that $I \times J = 0$, which is impossible unless $I = \pm J$. Thus $\chi(f) = \alpha \pm I \beta \in \bC_I$, which proves the lemma.

\end{proof}

Throughout this section, for a quaternion $q = x + yI \in \bC_I$ for some $I \in \bS$ and $x, y \in \bR$, we write, for each $J \in \bS$, $[q]_J$, the quaternion $x + Jy \in \bC_J$. It is clear that $q = [q]_I$ and $\bar{q} = x - Iy = [q]_{-I}$. If $y = 0$, then $[q]_J \in \bR$ for any $J \in \bS$.

Let $D$ be a symmetric slice domain in $\bH$. Let $c \in D$ and $I \in \bS$. We define the map $\chi_{c, J} : \cO_s(D) \to \bH$ by setting $\chi_{c, J},(f) = [f(c)]_J$ for every $f \in \cO_s(D)$. Such a map $\chi_{c, J}$ is called an \textit{evaluation} at the pair $(c, J) \in D \times \bS$. Thus for each $c \in D$, there is a sequence of evaluations $\left(\chi_{c, J}\right)_{J \in \bS}$ associated to $c$. By definition, the image of $\chi_{c, J}$, say $\chi_{c, J}(\cO_s(D))$, is a subset of $\bC_J$. 
\begin{lemma}
\label{lem-evaluation-character-Bers}

$\chi_{c, J}$ is a nontrivial character of $\cO_s(D)$ for every $c \in D$ and $J \in \bS$.

\end{lemma}

\begin{proof}

Condition (ii) in Definition \ref{def-char} is trivial. It suffices to verify conditions (i) and (iii) in Definition \ref{def-char}. Indeed, there exists an element $I \in \bS$ such that $c \in \bC_I$. Let $f, g \in \cO_s(D)$. Since $c \in \bC_I$, $f(c), g(c)$ belong in $\bC_I$, and thus $f(c) = u_1 + Iv_1$ and $g(c) = u_2 + Iv_2$ for some $u_1, v_1, u_2, v_2 \in \bR$. Thus $f(c) + g(c) = u_1 + u_2 + I(v_1 + v_2)$ and $f(c)g(c) = u_1u_2 - v_1v_2 + I(u_1v_2 + u_2v_1)$, and therefore
\begin{align*}
\chi_{c, J}(f + g) &= [f(c) + g(c)]_J \\
&= u_1 + u_2 + J(v_1 + v_2) \\
&= u_1 + Jv_1 + u_2 + Jv_2 \\
&= [f(c)]_J + [g(c)]_J \\
&= \chi_{c, J}(f) + \chi_{c, J}(g),
\end{align*}
which proves (i) of Definition \ref{def-char}.

For (iii) of Definition \ref{def-char}, we see that
\begin{align*}
\chi_{c, J}(fg) &= [f(c)g(c)]_J \\
&= u_1u_2 - v_1v_2 + J(u_1v_2 + u_2v_1) \\
&= (u_1 + Jv_1)(u_2 + Jv_2) \\
&= \chi_{c, J}(f) \chi_{c, J}(g).
\end{align*}

\end{proof}

For each $c \in D$, let $\cX_c : D \to \bH$ be the map defined by $\cX_c(f) = f(c)$. Such a map is called an \textit{evaluation at $c$}.

\begin{lemma}
\label{main-lem0-Bers}

Let $c \in D \cap \bC_I$ for some $I \in \bS$. Then $\chi_c = \chi_{c, I}$ and $\chi_c$ is a nontrivial character of $\cO_s(D)$.

\end{lemma}

\begin{proof}

Let $f \in \cO_s(D)$ be an arbitrary slice preserving regular function on $D$. Since $c \in \bC_I$, $f(c) \in \bC_I$. Write $f(c) = u + Iv$ for some $u, v \in \bR$. Thus $f(c) = u + Iv = [f(c)]_I$ for every $f \in \cO_s(D)$. Thus $\chi_c(f) = f(c) = [f(c)]_I = \chi_{c, I}(f)$ for every $f \in \cO_s(D)$, and thus $\chi_c = \chi_{c, I}$. It follows from Lemma  \ref{lem-evaluation-character-Bers} that $\chi_c$ is a nontrivial character of $\cO_s(D)$.

\end{proof}

\begin{lemma}
\label{main-lem1-Bers}

Let $\chi$ be a nontrivial character of $\cO_s(D)$, and $id_D : D \to \bH$ is the identity map which sends each $q \in D$ to itself. Set $c = \chi(id_D)$. Then
\begin{itemize}

\item [(i)] $c \in D$; and

\item [(ii)] $\chi = \chi_{c}$.

\end{itemize}

\end{lemma}

\begin{proof}

It is clear that $q - c$ and $q - \bar{c}$ are regular functions on $D$. Thus the function $e : D \to \bH$ defined by
\begin{align*}
e(q) = (q - c) \star (q - \bar{c}) = q^2 - q2\Re(c) + |c|^2
\end{align*}
is regular on $D$. Since $\Re(c)$ and $|c|^2$ are real numbers, $e(q) \in \cO_s(D)$.

Since $q^2 = id_D(q)^2$, and $id_D \in \cO_s(D)$, we see from Remark \ref{rem-nontrivial-char}(ii) that
\begin{align*}
\chi(e) = \chi(id_D)^2 - \chi(id_D)2\Re(c) + |c|^2.
\end{align*}

Write $c = a + Ib$ for some $a, b \in \bR$ and $I \in \bS$. We deduce from the above equation that 
\begin{align*}
\chi(id_D)^2 - \chi(id_D)2\Re(c) + |c|^2 &= c^2 - c2\Re(c) + |c|^2 \\
&= a^2 + 2abI - b^2 - (a + Ib)2a + a^2 + b^2 = 0,
\end{align*}
and thus 
\begin{align}
\label{e0-main-lem1-Bers}
\chi(e) = 0. 
\end{align}

From the equation of $e$, we know that the zero set of $e$ in $D$ is 
\begin{align*}
\cZ(e) = \{ [c]_J \; | \; J \in \bS\} \cap D = [x + \bS y] \cap D.
\end{align*}

Suppose that $c \not\in D$. Since $D$ is symmetric, $D \cap [x + \bS y] = \emptyset$, and thus the zero set of $e$ is empty. Therefore by Lemma \ref{lem-inverse-of-reg-func-in-O_s}, the function $f : D \to \bH$ defined by $f(q) = \dfrac{1}{e(q)}$ belongs in $\cO_s(D)$, and thus
\begin{align*}
1 = \chi(\b1) = \chi(ef) = \chi(e) \chi(f) = 0,
\end{align*}
which is a contradiction. Therefore $c \in D$. 

Take an arbitrary $f \in \cO_s(D)$. Let $P : D \to \bH$ be the function defined by
\begin{align*}
P(q) = f(q)^2 - f(q)2 \Re(f(c)) + |f(c)|^2
\end{align*}
for each $q \in D$. Since $f \in \cO_s(D)$ and $\Re(f(c)), |f(c)|^2$ are real numbers, we deduce that $P \in \cO_s(D)$. Thus
\begin{align}
\label{e1-main-lem1-Bers}
\chi(P) = \chi(f)^2 - \chi(f)2\Re(f(c)) + |f(c)|^2.
\end{align} 

One can check that
\begin{align*}
P(c) = f(c)^2 - f(c)2 \Re(f(c)) + |f(c)|^2 = 0,
\end{align*}
and thus $q = c \in D$ is a zero of $P(q)$. Since $P$ is a slice preserving regular function on $D$ and $c = a + Ib$, it follows from Lemma \ref{lem-zeros-of-slice-preserving-regular-func} that the sphere $[a + \bS b]$ is a spherical zero of $P$, and thus
\begin{align*}
P(q) = ((q - a)^2 + b^2)Q(q)
\end{align*}
for some slice preserving regular function $Q \in \cO_s(D)$. Note that 
\begin{align*}
e(q) = q^2 - q2\Re(c) + |c|^2 = (q - a)^2 + b^2,
\end{align*}
and thus
\begin{align*}
P(q) = e(q)Q(q).
\end{align*}
It follows from (\ref{e0-main-lem1-Bers}) and (\ref{e1-main-lem1-Bers}) that
\begin{align*}
 \chi(f)^2 - \chi(f)2\Re(f(c)) + |f(c)|^2 = \chi(P) = \chi(eQ) = \chi(e)\chi(Q) = 0,
 \end{align*}
 and so $\chi(f)$ is a zero to the polynomial $q^2 - q2\Re(f(c)) + |f(c)|^2 = ((q - \Re(f(c))^2 + \Im(f(c))^2$.

We know that the zero set of $q^2 - q2\Re(f(c)) + |f(c)|^2 = ((q - \Re(f(c))^2 + \Im(f(c))^2$ is the sphere $[\Re(f(c)) + \bS \Im(f(c))]$, and thus
\begin{align}
\label{e2-main-lem1-Bers}
\chi(f) = \Re(f(c)) + L_f\Im(f(c)) = [f(c)]_{L_f} \in \bC_{L_f}
\end{align}
for some $L_f \in \bS$. 

We consider the following cases.

\textit{Case 1. $c \in \bR$.}

For any slice preserving regular function $f$ in $\cO_s(D)$, $f(D \cap \bC_I) \subset \bC_I$ for every $I \in \bS$. Since $D$ is a symmetric slice domain, $D \cap \bR$ is nonempty. Note that for any $I, J \in \bS$ with $I$ being orthogonal to $J$ as vectors in $\bR^4$, $\bR = \bC_I \cap \bC_J$, and thus
\begin{align*}
(D \cap \bC_I) \cap (D \cap \bC_J) = D \cap (\bC_I \cap \bC_J) = D \cap \bR \ne \emptyset.
\end{align*}
Thus for any $f \in \cO_s(D)$, $f( D \cap \bR)$ is a subset of $\bC_I \cap \bC_J$ which is the reals $\bR$, and therefore
\begin{align*}
f(D \cap \bR) \subset \bR.
\end{align*}

Since $f(c) \in \bR$ for every $f \in \cO_s(D)$, we deduce from (\ref{e2-main-lem1-Bers}) that $\chi(f) = [f(c)]_{L_f} = f(c) = \chi_c(f)$, and thus $\chi = \chi_c$, where $c = \chi(id_D)$.

\textit{Case 2. $c \not\in \bR$.}

Since $c = a + Ib$, $b \ne 0$, and thus $c \in \bC_I \setminus \bR$. Since $c = \chi(id_D) \in \bC_I \setminus \bR$, it follows from Lemma \ref{main-lem1-Bers} that the image of $\chi$ is a subset of $\bC_I$, i.e., $\chi(f) \in \bC_I$ for every $f \in \cO_s(D)$/ Thus we deduce immediately from (\ref{e2-main-lem1-Bers}) that for every $f \in \cO_s(D)$, either $L_f = I$ or $L_f = -I$. Thus, for every $f \in \cO_s(D)$, either
\begin{align}
\label{e4-main-lem1-Bers}
\chi(f) = [f(c)]_I
\end{align}
or
\begin{align}
\label{e5-main-lem1-Bers}
\chi(f) = [f(c)]_{-I}
\end{align}

Suppose that there exists  $f \in \cO_s(D)$ such that $\chi(f) = [f(c)]_{-I}$ and $f(c) \not\in \bR$. Since $\chi$ is a character of $\cO_s(D)$, we deduce that
\begin{align}
\label{e3-main-lem1-Bers}
\chi(f id_D) = \chi(f)\chi(id_D) = [f(c)]_{-I}c.
\end{align}

We know that either $\chi(f id_D) = [f(c) id_D(c)]_I = [f(c)c]_I$ or $\chi(f id_D) = [f(c) id_D(c)]_{-I} = [f(c) c]_{-I}$. 

If $\chi(f id_D) = [f(c)c]_I$, then it follows from (\ref{e3-main-lem1-Bers}) that 
\begin{align*}
[f(c)]_{-I}c =  [f(c)]_I [c]_I = [f(c)]_I c.
\end{align*}
Since $c \ne 0$, $[f(c)]_{-I} = [f(c)]_I$, which implies that $f(c) \in \bR$, a contradiction. 

If $\chi(f id_D) = [f(c)c]_{-I}$, then it follows from (\ref{e3-main-lem1-Bers}) that 
\begin{align*}
[f(c)]_{-I}c =  [f(c)]_{-I} [c]_{-I} = [f(c)]_{-I} \bar{c}.
\end{align*}
Since $f(c) \ne \bR$, we deduce that $[f(c)]_{-I} \ne 0$, and thus $c = \bar{c}$, which implies that $c \in \bR$, a contradiction. 

Thus for every $f \in \cO_s(D)$, $\chi(f) = [f(c)]_I$ or $f(c) \in \bR$. Note that if $f(c) \in \bR$, $f(c) = [f(c)]_J = [f(c)]_I$ for any $J \in \bS$, and thus we deduce from Lemma \ref{main-lem0-Bers}, (\ref{e4-main-lem1-Bers}) and (\ref{e5-main-lem1-Bers}) that $\chi(f) = [f(c)]_I = \chi_{c, I}(f) = \chi_c(f)$. Therefore $\chi = \chi_c$, which proves the lemma.

\end{proof}

We now prove a quaternionic analogue of Bers's theorem (see the classical Bers theorem in Bers \cite{bers} or Remmert \cite{rem2}).

\begin{theorem}
\label{Bers-thm}

Let $D, \widehat{D}$ be symmetric slice domians in $\bH$. For every $\bR$-algebra homomorphism $\phi : \cO_s(D) \to \cO_s(\widehat{D})$, there exists a unique map $h : \widehat{D} \to D$ such that $\phi(f) = f \circ h$ for all $f \in \cO_s(D)$. More precisely,
\begin{itemize}

\item [(i)] $h = \phi(id_D) \in \cO_s(\widehat{D})$, where $id_D$ is the identiy map defined by $id_D(q) = q$ for every $q \in D$; and

\item [(ii)] $\phi$ is bijective if and only if $h$ is bi-regular. 

\end{itemize}

\end{theorem}

\begin{proof}

If a map $h : \widehat{D} \to D$ satisfies the theorem, then $\phi(f) = f \circ h$ for all $f \in \cO_s(D)$. In particular, $\phi(id_D) = id_D \circ h = h$, which implies the uniqueness of $h$. So it suffices to verify that the theorem does in fact hold for the map $h := \phi(id_D) \in \cO_s(\widehat{D})$. Note that for each $a \in \widehat{D}$, $\chi_a \circ \phi$ is a character of $\cO_s(D)$. By Lemma \ref{main-lem1-Bers}, 
\begin{align*}
\chi_a \circ \phi = \chi_{c},
\end{align*}
where 
\begin{align*}
c = \chi_a \circ \phi(id_D) = \chi_a(\phi(id_D)) = \chi_a(h) = h(a) \in D.
\end{align*}

Letting $a$ range over $c = \chi(id_D)$, and repeating the same arguments as above, we deduce that the image of $h$ is a subset of $D$, i.e., $h  : \widehat{D} \to D$. Furthermore,
\begin{align*}
\chi_a \circ \phi = \chi_{c} = \chi_{h(a)},
\end{align*}
and thus for all $f \in \cO_s(D)$ and $a \in \widehat{D}$, 
\begin{align}
\label{e1-Bers-thm}
\phi(f)(a) = \chi_a(\phi(f)) = (\chi_a \circ \phi)(f) = \chi_{h(a)}(f) = f(h(a)) = (f \circ h)(a).
\end{align}
 Therefore 
\begin{align*}
\phi(f) = f \circ h
\end{align*}
for all $f \in \cO_s(D)$.

We prove part (ii). Suppose that $\phi$ is bijective, and $\phi^{-1} : \cO_s(\widehat{D}) \to \cO_s(D)$ is its inverse. Following the same arguments as above, we deduce that the map $g := \phi^{-1}(id_{\widehat{D}}) \in \cO_s(D)$ satisfies $\phi^{-1}(f) = f \circ g$ for all $f \in \cO_s(\widehat{D})$.  We claim that $g$ is the inverse of $h$. Indeed, since $g \in \cO_s(D)$, we know from the construction of $h$ that
\begin{align*}
g \circ h = \phi(g) = \phi(\phi^{-1}(id_{\widehat{D}})) = id_{\widehat{D}}.
\end{align*}
Exchanging the roles of $g$ and $h$, we also obtain that $h \circ g = id_D$, which proves that $h$ is bi-regular. 

Suppose now that $h$ is bi-regular. Let $g : D \to \widehat{D} $ be the inverse of $h$. Let $\psi : \cO_s(\widehat{D}) \to \cO_s(D)$ be the map defined by 
\begin{align*}
\psi(f) = f \circ g
\end{align*}
for all $f \in \cO_s(\widehat{D})$. It is immediate that $\psi$ is an $\bR$-algebra homomorphism, and the inverse of $\phi$. Hence the theorem follows.

\end{proof}

\section{A quaternionic analogue of Iss'sa's theorem}
\label{sec-isssa}

In this section, we prove a quaternionic analogue of Iss'sa's theorem \footnote{Hej Iss'sa is the pseudonym of a renown Japanese mathematician} (see \cite{isssa} or \cite{rem2}) which is a generalization of Bers's theorem to function fields. For a beautiful exposition of the classical Iss'sa's theorem in $\bC$, the reader is referred to \cite{rem2}.

Let $D$ be a symmetric slice domain. Recall that $\cQ(\cO_s(D))$ is the quotient field of $\cO_s(D)$ (see Section \ref{sec-divisors}). Let $\cQ(\cO_s(D))^{\times}$ be the multiplicative (abelian) group $\cQ(\cO_s(D)) \setminus \{0\}$. A map $v :\cQ(\cO_s(D))^{\times} \to \bZ$ is called a \textit{valuation on $\cQ(\cO_s(D))$} if for every $f, g \in\cQ(\cO_s(D))^{\times}$, 
\begin{itemize}

\item[(V1)] $v(fg) = v(f) + v(g)$; and

\item [(V2)] $v(f + g) \ge \min(v(f), v(g))$ if $f \ne -g$.

\end{itemize}

By declaring $v(f) = \infty$ if and only if $f \equiv 0$, we can extend $v$ to the whole field $\cQ(\cO_s(D))$. 

\begin{lemma}
\label{main-lem0-Iss'sa}

Let $v$ be a valuation on $\cQ(\cO_s(D))$. Then $v(c) = 0$ for all quaternions $c \in \bH \setminus \{0\}$.

\end{lemma}

\begin{proof}

By the Fundamental Theorem of Algebra for polynomials over the quaternions $\bH$ (see, for example, \cite{en, gsv, niven, ps}), for every integer $n \ge 1$, there exists $\alpha_n \in \bH \setminus \{0\}$ such that $\alpha_n^n = c$. It follows from (V1) that $v(c) = nv(\alpha_n) \in n\bZ$ for all $n \ge 1$, which implies that $n$ divides $v(c)$ for every $n \ge 1$. This is possible only if $v(c) = 0$.

\end{proof}

\begin{remark}
\label{rem-Iss'sa}

\begin{itemize}

\item []

\item [(i)] Condition (V2) immediately implies the following:
\begin{itemize}

\item [(V2')] $v(f + g) = \min(v(f), v(g))$ if $v(f) \ne v(g)$. 

\end{itemize}
\item [(ii)] Condition (V1) implies that $v(1) = 0$, and thus $v(1/f) = - v(f)$.

\end{itemize}
The proof of this result is exactly the same as in the classical case for valuations on meromorphic functions on domains in the complex plane $\bC$ (see \cite[p.109]{rem2}).

\end{remark}

\begin{example}
\label{example-Iss'sa}

The natural valuations on $\cQ(\cO_s(D))$ are the order functions $\sord_{(\cdot)}(c)$ with $c \in D$ introduced in Section \ref{sec-divisors} which associate to each slice preserving semiregular function $f \in\cQ(\cO_s(D))^{\times}$ its spherical order at the point $c$, say $\sord_f(c) \in \bZ$. The following is a quaternionic analogue of the holomorphy criterion in the classical complex analysis (see \cite[p.109]{rem2}).

\textbf{Holomorphy Criterion.} \textit{A slice preserving semiregular function $f \in \cQ(\cO_s(D))^{\times}$ is regular if and only if $\sord_f(c) \ge 0$ for all $c \in D$. Equivalently, $\divi(f) \ge 0$}

The proof of the above result follows immediately from that of Lemma \ref{lem-holomorphy-criterion}.

\end{example}

\begin{lemma}
\label{main-lem1-Iss'sa}

Let $v$ be a valuation on $\cQ(\cO_s(\bH))$. Then $v(q) \ge 0$.

\end{lemma}

\begin{proof}

Assume the contrary, i.e., $v(q) = -m$ for some positive integer $m \ge 1$. Since $v(c) = 0$ for all $c \in \bH \setminus \{0\}$, we deduce that
\begin{align}
\label{e0-main-lem1-Iss'sa}
v(q - c) = \min (v(q), v(c)) = -m
\end{align}
for all $c \in \bH \setminus \{0\}$. 

Now take an arbitrary integer $d \ge 2$. Let $a_n = n$ for each positive integer $n \in \bZ_{> 0}$. By Lemma \ref{Weierstrass-lem}, the function 
\begin{align}
\label{e0-main-lem1-Iss'sa}
\cP_{(a_n)_{n \in \bZ_{> 0}}, d}(q) =  \prod_{n \in \bZ_{> 0}} (1 - qa_n^{-1})^{d^n}g_{n, a_n}(q)^{d^n}
\end{align}
has no zeros in $\bH \setminus \bZ_{\ge 0}$ and vanishes to order $d^n$ at each $a_n = n \in \bZ_{> 0}$, where the $g_{n, a_n}$ are the convergence-producing regular factors defined in Definition \ref{def-regular-factor}.

For each integer $\ell \ge 2$, set
\begin{align*}
\cP_{\ell}(q) = \left(\prod_{h = 1}^{\ell - 1} (1 - qa_h^{-1})^{d^h}g_{h, a_h}(q)^{d^h}\right)^{-1}\cP_{(a_n)_{n \in \bZ_{\ge 0}}, d}(q),
\end{align*}
\begin{align}
\label{e1/2-main-lem1-Iss'sa}
\cH_{\ell}(q) = \left(\prod_{h = 1}^{\ell - 1} (1 - qa_h^{-1})^{d^h}\right)^{-1}\cP_{(a_n)_{n \in \bZ_{\ge 0}}, d}(q),
\end{align}
and
\begin{align*}
\cQ_{\ell}(q) = \prod_{n \ge \ell} (1 - qa_n^{-1})^{d^{n - \ell}}g_{n, a_n}(q)^{d^{n - \ell}}.
\end{align*}

Since the $(1 - qa_h^{-1})^{d^h}g_{h, a_h}(q)^{d^h}$ are slice preserving regular functions in $\cO_s(\bH)$ (see the remark following Definition \ref{def-regular-factor}), the first $\ell - 1$ factors in $\cP_{\ell}$ and $\cH_{\ell}$ commutes with each other. Thus
\begin{align*}
\cP_{\ell}(q) = \prod_{n = \ell}^{\infty} (1 - qa_n^{-1})^{d^n}g_{n, a_n}(q)^{d^n},
\end{align*}
\begin{align*}
\cH_{\ell}(q) = \left(\prod_{h = 1}^{\ell - 1} g_{h, a_h}(q)^{d^h}\right) \prod_{n = \ell}^{\infty} (1 - qa_n^{-1})^{d^n}g_{n, a_n}(q)^{d^n}.
\end{align*}
and
\begin{align}
\label{e1/4-main-lem1-Iss'sa}
\cH_{\ell}(q) = \left(\prod_{h = 1}^{\ell - 1} g_{h, a_h}(q)^{d^h}\right)\cP_{\ell}(q).
\end{align}

By the equations above, and Lemma \ref{Weierstrass-lem}, $\cH_{\ell}$, $\cP_{\ell}$, $\cQ_{\ell}$ all belong in $\cO_s(\bH)$.

We see from Remark \ref{rem-Iss'sa} and (\ref{e1/2-main-lem1-Iss'sa}) that
\begin{align}
\label{e1-main-lem1-Iss'sa}
v(\cH_{\ell}) =  m \sum_{h = 0}^{\ell - 1} d^h + v(\cP_{(a_n)_{n \in \bZ_{\ge 0}}, d})=   \dfrac{m}{d - 1}(d^{\ell} - 1) + v(\cP_{(a_n)_{n \in \bZ_{\ge 0}}, d}).
\end{align}

On the other hand, since the polynomial $q^{d^{\ell}}$ is a slice preserving regular function, and for all $h \ge \ell$, the function $\cQ_{\ell}^{(h)}(q) =  \prod_{n = \ell}^h (1 - qa_n^{-1})^{d^{n - \ell}}g_{n, a_n}(q)^{d^{n - \ell}}$ is a slice preserving regular function, we deduce from Lemma \ref{lem-comp-of-slice-reg-func} that $(\cQ_{\ell}^{(h)}(q))^{d^{\ell}}$ is a slice preserving regular function in $\bH$, and it thus follows from Lemma \ref{Weierstrass-lem} that
\begin{align*}
\cQ_{\ell}(q)^{d^{\ell}} &= \left(\lim_{h \to \infty}\cQ_{\ell}^{(h)}(q)\right)^{d^{\ell}} \\
&= \lim_{h \to \infty}\left(\cQ_{\ell}^{(h)}(q)\right)^{d^{\ell}} \\
&= \lim_{h \to \infty}\prod_{n = \ell}^h (1 - qa_n^{-1})^{d^n}g_{n, a_n}(q)^{d^n} \\
&=\prod_{n \ge \ell} (1 - qa_n^{-1})^{d^n}g_{n, a_n}(q)^{d^n} \\
&= \cP_{\ell}(q).
\end{align*}

By Lemma \ref{lem-nth-roots-of-exp} and the definition of $ g_{h, a_h}$ (see  Definition \ref{def-regular-factor}), for every integer $h$ with $1 \le h \le \ell - 1$, there exists a slice preserving regular function $\cG_{h, a_h}$ such that
\begin{align*}
\cG_{h, a_h}(q)^{d^{\ell}}= g_{h, a_h}(q).
\end{align*}
Since the $G_{h, a_h}$ are slice preserving, they commute with each other, and thus
\begin{align*}
\left(\prod_{h = 1}^{\ell - 1} G_{h, a_h}(q)^{d^h}\right)^{d^{\ell}}   =\prod_{h = 1}^{\ell - 1} g_{h, a_h}(q)^{d^h}.
\end{align*}
Since $G_{h, a_h}, \cQ_{\ell}$ are in $\cO_s(\bH)$, $G_{h, a_h} \cQ_{\ell} = \cQ_{\ell} G_{h, a_h}$, and thus we deduce  from (\ref{e1/4-main-lem1-Iss'sa}) that
\begin{align*}
 \cH_{\ell}(q) =\left(\prod_{h = 1}^{\ell - 1} g_{h, a_h}(q)^{d^h}\right) \cP_{\ell}(q) = \left(\cQ_{\ell}(q) \prod_{h = 1}^{\ell - 1} G_{h, a_h}(q)^{d^h}\right)^{d^{\ell}}.
\end{align*}

Therefore, 
\begin{align*}
v(\cH_{\ell}) =v\left( \left(\cQ_{\ell}(q) \prod_{h = 1}^{\ell - 1} G_{h, a_h}(q)^{d^h}\right)^{d^{\ell}} \right) = d^{\ell}v\left(\cQ_{\ell}(q) \prod_{h = 1}^{\ell - 1} G_{h, a_h}(q)^{d^h}\right)
\end{align*}
and thus 
\begin{align}
\label{e2-main-lem1-Iss'sa}
v(\cH_{\ell}(q)) \equiv 0 \pmod{d^{\ell}}.
\end{align}

From (\ref{e1-main-lem1-Iss'sa}) , (\ref{e2-main-lem1-Iss'sa}), we deduce that
\begin{align*}
\dfrac{m}{d - 1}(d^{\ell} - 1) + v(\cP_{(a_n)_{n \in \bZ_{\ge 0}}, d}) = v(\cH_{\ell}(q)) \equiv 0 \pmod{d^{\ell}}, 
\end{align*}
and thus
\begin{align*}
 v(\cP_{(a_n)_{n \in \bZ_{\ge 0}}, d}) \equiv \dfrac{m}{d - 1} \pmod{d^{\ell}}
\end{align*}
for all $\ell \ge 2$. Therefore $(d - 1)v(\cP_{(a_n)_{n \in \bZ_{\ge 0}}, d}) - m \in d^{\ell}\bZ$ for all $\ell \ge 2$, which is possible only when $ v(\cP_{(a_n)_{n \in \bZ_{\ge 0}}, d}) = \dfrac{m}{d - 1}$. Since $d \ge 2$ is arbitrary and $v(\cP_{(a_n)_{n \in \bZ_{\ge 0}}, d}) \in \bZ$, we deduce that $m = 0$, which is a contradiction since $m$ is a positive integer. This contradiction establishes that $v(q) \ge 0$.

\end{proof}

\begin{corollary}
\label{main-cor-Iss'sa}

Let $v$ be a valuation on $\cQ(\cO_s(D))$. Then $v(f) \ge 0$ for all $f \in \cO_s(D) \setminus \{0\}$.

\end{corollary}

\begin{proof}

If $f$ is constant, then Lemma \ref{main-lem0-Iss'sa} implies that $v(f) = 0$. 

Suppose that $f$ is nonconstant. Take any slice preserving semiregular function $g \in \cQ(\cO_s(\bH))^{\times}$ on the quaternions $\bH$. Then there exists slice preserving entire functions $\alpha, \beta \in \cO_s(\bH)$ such that $\alpha, \beta \not\equiv 0$ and $g = \dfrac{\alpha}{\beta}$. Since $f$ is a slice preserving regular function on $D$, we deduce from Lemma \ref{lem-comp-of-slice-reg-func} that both $\alpha \circ f : D \to \bH$ and $\beta \circ f : D \to \bH$ are slice preserving regular functions on $D$. 

We show that $\beta \circ f \not\equiv 0$. Assume the contrary, i.e., $\beta \circ f \equiv 0$, which implies that $\beta(f(q)) = 0$ for all $q \in D$. 

We claim  that for all $I \in \bS$, the restriction $f_I$ of $f$ to $D \cap \bC_I$ is nonconstant. Assume the contrary, i.e., there exists an element $I \in \bS$ such that $f_I$ is constant, say $f(q) = c$ for all $q \in D \cap \bC_I$, where $c$ is some constant in $\bH$. Take any $x + y\bS \subset D$, and let $J$ be an arbitrary element in $\bS$. By the Representation Formula (see Theorem \ref{cor-1.16-gss}), we deduce that
\begin{align*}
f(x + yJ) &= \dfrac{1}{2}(f(x + yI) + f(x - yI)) + \dfrac{JI}{2}(f(x - yI) - f(x + yI)) = c,
\end{align*}
which proves that $f$ is constant on $D$, a contradiction. Thus the restriction $f_I$ of $f$ to $D \cap \bC_I$ is nonconstant for all $I \in \bS$. Since $f$ is a slice preserving regular function on $D$, $f_I$ is a holomorphic map from $D \cap \bC_I$ to $\bC_I$ for all $I \in \bS$. By the classical open mapping theorem (see, for example, \cite{rem1}), $U_I := f_I(D \cap \bC_I)$ is open in $\bC_I$ for all $I \in \bS$.

Since $\beta(f(q)) = 0$ for every $q \in D$, the restriction to $\beta$ to $U_I$ is identically to zero for every $I \in \bS$, and therefore by the Identity Theorem (see Theorem \ref{thm1.12-gss}), $\beta \equiv 0$, a contradiction. Thus $\beta \circ f \not\equiv 0$.

Repeating the same arguments as above, we deduce that $\alpha \circ f \not\equiv 0$. Thus $g \circ f = \dfrac{\alpha \circ f }{\beta \circ f}$ is a slice preserving semiregular function in the quotient field $\cQ(\cO_s(D))^{\times}$. 

Thus the map $v_f :  \cQ(\cO_s(\bH))^{\times} \to \bZ$ defined by assigning to each function $g \in \cQ(\cO_s(\bH))^{\times}$ the integer $v(g \circ f)$ is well-defined, and it is clear that $v_f$ is a valuation on $ \cQ(\cO_s(\bH))$. By Lemma \ref{main-lem1-Iss'sa}, 
\begin{align*}
v(f) = v_f(q) \ge 0,
\end{align*}
which proves our corollary.

\end{proof}

\begin{lemma}
\label{main-lem2-Iss'sa}

Let $D, \widehat{D}$ be symmetric slice domains in $\bH$. For every $\bR$-algebra homomorphism $\phi :  \cQ(\cO_s(D)) \to  \cQ(\cO_s(\widehat{D}))$, 
\begin{align*}
\phi(\cO_s(D)) \subset \cO_s(\widehat{D}).
\end{align*}

\end{lemma}

\begin{proof}

Since $\phi$ is an $\bR$-algebra homomorphism and both $\cQ(\cO_s(D))$ and $\cQ(\cO_s(\widehat{D}))$ are fields, its kernel is trivial, and thus $\phi$ is injective. Therefore $\phi(f) \ne 0$ for all $f \in \cQ(\cO_s(D))^{\times}$. For every point $c \in \widehat{D}$, the map $v_c : \cQ(\cO_s(D))^{\times} \to \bZ$ defined by
\begin{align*}
v_c(f) := \sord_{\phi(f)}(c), \; \; f \in \cQ(\cO_s(D))^{\times},
\end{align*}
defines a valuation on $\cQ(\cO_s(D))$. Here $\sord_{(\cdot)}(c) : \cQ(\cO_s(\widehat{D}))^{\times} \to \bZ$ is the valuation in Example \ref{example-Iss'sa}, introduced in Section \ref{sec-divisors}. By Corollary \ref{main-cor-Iss'sa}, for all $f \in \cO_s(D)$, $v_c(f) \ge 0$, and thus $\sord_{\phi(f)}(c) \ge 0$ for all $c \in \widehat{D}$. By Holomorphy Criterion following Example \ref{example-Iss'sa}, $\phi(f) \in \cO_s(\widehat{D})$ for all $f \in \cO_s(D)$.

\end{proof}

Combining Theorem \ref{Bers-thm} and the above lemma, and noting that $\cQ(\cO_s(D))$ is the quotient field of the ring $\cO_s(D)$, the following result, viewed as a quaternionic analogue of Iss'sa's theorem (see Iss'sa \cite{isssa} or \cite[Iss'sa's theorem, p.109]{rem2}), is immediate.

\begin{theorem}
\label{Iss'sa-thm}

Let $\phi : \cQ(\cO_s(D)) \to \cQ(\cO_s(\widehat{D}))$ be any $\bR$-algebra homomorphism. Then there exists exactly one regular map $h : \widehat{D} \to D$ such that $\phi(f) = f \circ h$ for all $f \in \cQ(\cO_s(D))$.

\end{theorem}

In view of Theorem \ref{thm-M_s(H)-quotient-field}, we obtain the following.
\begin{theorem}
\label{Iss'sa-thm-for-H}

Let $\phi : \cM_s(\bH) \to \cM_s(\bH)$ be any $\bR$-algebra endomorphism of $\cM_s(\bH)$. Then there exists exactly one entire function $h : \bH \to \bH$ such that $\phi(f) = f \circ h$ for all $f \in \cM_s(\bH)$.

\end{theorem}

\end{document}